\documentclass[makeidx]{amsart}
\usepackage{latexsym}
\usepackage{graphicx}
\usepackage{subfigure}
\usepackage{amsmath}
\usepackage{amssymb}
\usepackage{mathrsfs}
\usepackage{epstopdf}
\newtheorem{theorem}{Theorem}[section]
\newtheorem{lemma}[theorem]{Lemma}
\newtheorem{prop}[theorem]{Proposition}
\newtheorem{corollary}[theorem]{Corollary}
 \newtheorem{definition}[theorem]{Definition}
 \newtheorem{note}[theorem]{Note}

\newenvironment{remark}{\noindent \textbf{Remark}.}{\hfill $\square$}

\renewcommand{\Re}{\mathop{\rm Re}\nolimits}
\renewcommand{\Im}{\mathop{\rm Im}\nolimits}

\newcommand{\sgn}{\mathop{\rm sgn}\nolimits}

\numberwithin{equation}{section} \makeindex
\begin{document}
\author{ J. Ye$^1$ and S.  Tanveer$^2$  }
\title{Global solutions for two-phase
Hele-Shaw bubble for a near-circular initial shape}

\gdef\shortauthors{ J. Ye \& S. Tanveer  } \gdef\shorttitle{Global
solutions to the bubble}
\thanks{$1$.  Department of Mathematics, Ohio State University, Columbus, OH 43210 (jenny$_{-}$yeyj@math.ohio-state.edu).}
\thanks{$2$.  Department of Mathematics, Ohio State University, Columbus, OH 43210 (tanveer@math.ohio-state.edu).}
\maketitle

\today

\bigskip
\begin{abstract} Using a vortex sheet method we prove global
existence of a near circular initial bubble in a Hele-Shaw cell with
surface tension and generally finite nonzero viscosity ratio between
fluids inside and
outside the bubble. The circular shape is shown to be
asymptotically stable for all sufficiently smooth small perturbation.
The initial condition in this case, while smooth, need not be analytic.
\end{abstract}

\section{Introduction.}
The displacement of a more viscous fluid by a less viscous fluid in a
Hele-Shaw cell has been a problem of considerable physical as well
as mathematical interest. Over the years, many reviews have appeared
from a range of perspectives (Saffman \cite{PG2}; Bensimon {\it {et
al.}} \cite{BE1}; Homsy \cite{HO1}; Pelce \cite{PE1}; Kessler {\it
et al.} \cite{KE1}; Tanveer \cite{Tanveer3} and \cite{Tanveer1};
Hohlov \cite{HY1}; Howison \cite{HOW2} and \cite{HO}). An important
issue is the stability of steadily propagating shapes such as a
semi-infinite Saffman-Taylor finger \cite{PG1}, or a finite
translating bubble. Although there has been a lot of formal
asymptotic \cite{Tanveer2} as well as numerical computations
\cite{MC1}, some issues of stability have not been completely
resolved without controversy.

Rigorous mathematical tools for investigation of global solutions
and nonlinear stability problem are still quite limited. Thus far,
global stability of near circular analytic shapes has been
established by
P. Constantin and M. Pugh \cite{PM}
for one-phase Hele-Shaw bubble in the absence of pressure gradient
or injection of less viscous fluid.

The present  paper generalizes these results for nonzero viscosity
ratio between interior and exterior fluids for nonanalytic but
sufficient smooth shapes. Unlike \cite{PM}, where conformal mapping
is used, this paper relies on a vortex sheet equal-arclength
formulation, originally due to Hou, Lowengrub, and Shelley
\cite{HL1}. This is particularly advantageous for study of interface
motion between fluids with nonzero viscosities. Ambrose \cite{AD3}
used this formulation to prove local existence of Hele-Shaw solution
for general  initial shapes \cite{AD3} without surface tension;
earlier Duchon and Robert \cite{DR1} obtained local existence
results with surface tension in a differing formulation for
one-phase Hele-Shaw problem.

In this paper, Ambrose's approach has been extended suitably to
obtain global existence and stability results for smooth but non-analytic
near-circular initial shapes in two phase Hele-Shaw flow with surface tension.

In the equal arc-length vortex sheet formulation,
the boundary curve between the two fluids of differing viscosity is
described parametrically at any time $t$ by $z=x(\alpha,t)+iy(\alpha,t)$.
$\alpha$ is chosen so that
$z(\alpha+2
\pi,t)=z(\alpha,t)$. We define $\theta$ so that
$\alpha+\theta$ for the angle formed between the tangent to the curve and
the horizontal ($x$-axis), as the boundary is tranversed counter-clockwise
with increasing $\alpha$.
Hou, Lowengrub and Shelley in \cite{HL2} observed that
a
choice\footnote{This choice or any other choice of
tangential speed of points on the interface has no effect on the interface
shape itself.}
of the tangent velocity $T$
is possible so that $s_\alpha$ is independent of $\alpha$, where $s$ is
the arc-length. They also observed that this choice
simplifies the evolution equation for $\theta$.

It is convenient to introduce the map $\Phi:\mathbb{R}^2\rightarrow
\mathbb{C}$ by $\Phi(a,b)= a+ib$. Then
the velocity ${\bf W}$ (see \cite{HL1}) generated by a vortex sheet on
the boundary of strength $\gamma (\alpha)$ is given by the
Birkhoff-Rott integral which has
the complex representation:
\begin{eqnarray}
\label{1.1}
[\Phi({\bf W})]^{\ast}=\frac{1}{2\pi
i}PV\int_0^{2\pi}\frac{\gamma(\alpha')}{z(\alpha)-z(\alpha')}d\alpha'.
\end{eqnarray}
The unit tangent and normal vectors to the curve clearly satisfy
\begin{displaymath}
\Phi({\bf t})=\frac{2\pi z_\alpha}{L},\,\,\Phi({\bf n})=\frac{2\pi i
z_\alpha}{L}.
\end{displaymath}
The normal velocity $U(\alpha,t)$ of the curve is given by
\begin{equation}
\label{1.2}
U(\alpha,t)={\bf W}\cdot {\bf n}.
\end{equation}

The equations for the evolution of a Hele-Shaw interface in the
infinite domain with surface tension accounted for are given by (see
\cite{HL1})
 \begin{equation*}
 \left\{\begin{aligned}
\theta_t(\alpha,t)&=\frac{2\pi}{L}U_\alpha(\alpha,t)+\frac{2\pi}{L}T(\alpha,t)\big(1+\theta_\alpha(\alpha,t)\big),\\
L_t(t)&=-\int_0^{2\pi}\big(1+\theta_\alpha(\alpha,t)\big)
U(\alpha,t) d\alpha,
\end{aligned}\right.\tag{A.1}\end{equation*}
with
\begin{equation*}
\left\{\begin{aligned} \gamma(\alpha,t)&=-\frac{L}{\pi} A_{\mu} {\bf
W}\cdot {\bf
t}+\frac{2\pi}{L}\sigma \theta_{\alpha \alpha},\\
T(\alpha,t)&=\int_0^{\alpha}
\big(1+\theta_{\alpha'}(\alpha',t)\big)U(\alpha',t)
d\alpha'-\frac{\alpha}{2\pi} \int_0^{2\pi}
\big(1+\theta_{\alpha}(\alpha,t)\big)U(\alpha,t) d\alpha
,\end{aligned}\right.\tag{A.2}
\end{equation*}
where \begin{displaymath} A_{\mu}=\frac{\mu_1-\mu_2}{\mu_1+\mu_2},
\end{displaymath}
$\mu_1$ is the viscosity of the exterior fluid, $\mu_2$ is the
viscosity of the interior fluid, and $\sigma$ is the coefficient of
surface tension. The initial condition is given by
\begin{eqnarray}
\label{1.3} \theta(\alpha,0)=\theta_0(\alpha),\,\,L(0)=2\pi.
\end{eqnarray} In order that $z_\alpha = \frac{L}{2\pi} \exp \left [ i \alpha +
i \theta \right ]$, the specified $\theta_0 (\alpha) $ must satisfy
the the consistency condition
\begin{equation}
\label{1.3.1} \int_0^{2 \pi} \exp \left [ i \alpha + i \theta_0
(\alpha) \right ] d\alpha = 0.
\end{equation}
\begin{definition}
\label{def1.1}
Let $s\geq0$. The Sobolev
space $H^s\big(\mathbb{T}[0,2\pi]\big)$ is the set of all
$2\pi$-periodic function $f=\sum_{-\infty}^{\infty}\hat{f}(k)
e^{ik\alpha}$ such that
\begin{displaymath}
\|f\|_s=\sqrt{\sum_{k=-\infty}^{\infty}|k|^{2s}|\hat{f}(k)|^2+|\hat{f}(0)|^2}<\infty.
\end{displaymath}
\end{definition}
\begin{note}
\label{note1.1} For $f, g \in H^s
\left ( \mathbb{T} [0, 2 \pi ] \right ) $, the Banach Algebra property
$ \| f g \|_s \le C_s \| f \|_s \| g \|_s $ for $s \ge
1 $ for some constant $C_s$ depending on $s$ is easily proved and will
be useful in the sequel.
\end{note}
\begin{definition}
\label{def1.2}
The Hilbert transform, $\mathcal{H}$, of a function
$f \in H^0 \left ( \mathbb{T} [0, 2\pi] \right )$ with Fourier Series
$f=\sum_{-\infty}^{\infty}\hat{f}(k) e^{ik\alpha}$ is given by
\begin{eqnarray}
\mathcal{H}[f] (\alpha)&=&\frac{1}{2
\pi}PV\int_{0}^{2\pi}f(\alpha')\cot{\frac{1}{2}(\alpha-\alpha')}d\alpha'\nonumber\\
&=&\sum_{k\neq 0}-i\sgn(k)\hat{f}(k) e^{ik\alpha}.\nonumber
\end{eqnarray}
\end{definition}

\begin{note}
\label{note1.3}
The Hilbert transform commutes with differentiation. We will
denote derivative with respect to $\alpha$, either by $D$ or subscript
$\alpha$. Also, for the sake of brevity of notation, the time $t$ dependence
will often be omitted, except where it might cause
confusion otherwise.
\end{note}

\begin{definition}
\label{def1.4} We define the operator $\Lambda$ to be a derivative
followed by the Hilbert transform: $\Lambda=\mathcal{H}D$.
\end{definition}
\begin{note}\label{note1.4}
In
the Fourier representation, we have $(\Lambda f)_k= |k|\hat{f}(k)$. This
implies that
\begin{displaymath} \Big(\int_0^{2 \pi}\big(
f^2+f\Lambda f\big)d\alpha\Big)^{1/2}
\end{displaymath}
is  equivalent  to $H^{1/2}\big(\mathbb{T}[0,2\pi]\big)$ norm for a
real-valued $2\pi$-periodic function $f$. Further, if ${\hat f} (0)
=0$, then it is easily seen that $\left ( \int_0^{2\pi} f \Lambda f
d\alpha \right )^{1/2} = \| f \|_{1/2} $. Note that operator
$\Lambda$ is self-adjoint in $H^{1/2}\big(\mathbb{T}[0,2\pi]\big)$.
\end{note}
\begin{definition}
\label{def1.5}  Following Ambrose \cite{AD3}, we define commutator
$$ [\mathcal{H},f]g =
\mathcal{H}(fg)-f\mathcal{H}(g).$$ We also define the linear
integral operator $\mathcal{K} [z]$, depending on $z$, as
$$ \left (  \mathcal{K} [z] f  \right ) (\alpha) = \frac{1}{2 \pi i} \int_{\alpha-\pi}^{\alpha+\pi} f(\alpha')
\left [ \frac{1}{z (\alpha)-z(\alpha') } - \frac{1}{2 z_\alpha (\alpha') } \cot \frac{1}{2} (\alpha-\alpha') \right ] d\alpha'.
$$
\end{definition}

\begin{remark} For $2 \pi$-periodic functions $f$ and $z$, it is clear that the upper and lower
limits of the integral above can be replaced by $a$ and $a+2\pi$ respectively for arbitrary $a$. Further,
in terms of operators $\left [ \mathcal{H}, \frac{1}{z_\alpha} \right ]$ and  $\mathcal{K}$,
we may express
${\bf W}$ in the following form
(see \cite{AD1}):
\begin{eqnarray}
\label{2.1} [\Phi({\bf W})]^{\ast}=\frac{1}{2i} \left [ \mathcal{H}
, \frac{1}{z_\alpha} \right ] \gamma + \frac{1}{2i z_\alpha}
\mathcal{H} \gamma + \mathcal{K}[z] \gamma.
\end{eqnarray}
\end{remark}

\begin{definition}
\label{def2.2}
We define a complex operator $\mathcal{G}[z]$, depending on $z$, so that
\begin{equation}
\label{2.5.0} \mathcal{G}[z]\gamma =
z_\alpha\Big[\mathcal{H},\frac{1}{z_\alpha}\Big]\gamma +2i
z_\alpha\mathcal{K}\big[z\big]\gamma.
\end{equation}
It is also convenient to define a related
real operator $\mathcal{F} [z]$, depending on $z$, so that
\begin{equation}
\label{2.5} \mathcal{F} [z] \gamma = \Re \left ( \frac{z_\alpha
(\alpha)}{\pi i}  PV \int_0^{2 \pi} \frac{\gamma (\alpha')
}{z(\alpha)-z(\alpha')} d \alpha' \right ) = \Re \Big ( \frac{1}{i}
\mathcal{G} [z] \gamma \Big ).
\end{equation}
\end{definition}
From the expressions for $U$ and ${\bf W} \cdot {\bf t}$, it follows that
\begin{equation}
\label{Ueqn} U = \frac{\pi}{L} \mathcal{H} [\gamma] + \frac{\pi}{L}
\Re \left ( \mathcal{G} [z] \gamma \right ),
\end{equation}
\begin{equation}
\label{Wteqn}
{\bf W} \cdot {\bf t} = \frac{\pi}{L}\mathcal{F} [z] \gamma.
\end{equation}

\begin{definition}
\label{def1.6}
We introduce the projection operator $\mathcal{Q}_1$ so that
\begin{displaymath}
\left [ \mathcal{Q}_1 f \right ] (\alpha)=
f(\alpha)-\hat{f}(0)-\hat{f}(1)e^{i\alpha}-\hat{f}(-1)e^{-i\alpha},
\end{displaymath}
where $f=\sum_{-\infty}^{\infty}\hat{f}(k) e^{ik\alpha}$.
In general, ${\hat .}$ symbol will be reserved for Fourier components.
Further, we will denote
$\tilde\theta=\mathcal{Q}_1\theta$.
\end{definition}
\begin{definition}
\label{def1.10} We define $\dot{H}^s$ as the subspace of
$H^s\left(\mathbb{T}[0,2\pi]\right)$ containing real-valued
functions so that $\phi\in\dot{H}^s$ implies
$\mathcal{Q}_1\phi=\phi$. Note $\|\phi\|_s=\|D^s\phi\|_0$ for
$s\ge1$.
\end{definition}

 The significantly new aspect of the present paper
include a vortex sheet formulation equivalent to the evolution
system (A.1)-(A.2) with initial condition (\ref{1.3}) that projects
away the neutral linear modes so that exponentially decay of the
remaining Fourier modes helps control small nonlinearity.  The
equivalent system involves the evolution of $\tilde\theta$,
$\hat\theta(0)$ and $L$, where $\hat\theta(1)$ and $\hat\theta(-1)$
are determined as complex functionals of $\tilde\theta$.

 We
analyze the evolution of $\tilde\theta$ and $L$. We first form
Galerkin approximation in a finite dimensional space for the two
evolution equations. We then show that solutions to these
approximate equations exist by using the Picard theorem for
differential equations in Banach spaces.

We then define  energy
$E(t)=\frac{1}{2}\big\|D_\alpha^r\tilde\theta(\cdot,t)\big\|_0^2$
for $r\geq4$. We estimate its growth and find that if
$\big\|\mathcal{Q}_1\theta_0\big\|_r$ is small enough, then there
exists a positive constant $A$, which for concreteness is chosen to
be $\frac{\sigma}{18}$, so that
\begin{eqnarray}
\label{1.4}
\frac{dE}{dt}\leq -AE.
\end{eqnarray}

The estimate (\ref{1.4}) holds for  Galerkin approximate equations
and is independent of the truncation $n$. The exponential decay
estimates on $E(t)$ implied by this inequality help continue the
solution of the Galerkin approximation to arbitrary time. Further
estimates show that $\left \{ {\tilde \theta}_n \right
\}_{n=2}^\infty $ form a Cauchy sequence in $\dot{H}^1 $, which is
used to show that that ${\tilde \theta}_n $ converges to a strong
solution ${\tilde \theta}$ of the original system, which also decays
exponentially. The functional relation determining ${\hat \theta}
(1)$ and ${\hat \theta} (-1)$ shows that $\theta -{\hat \theta} (0)$
also decays exponentially in time, which implies that circular
shapes are asymptotically stable.

The main result in this paper is the following Theorem:
\begin{theorem}
\label{Thm1.7} There exists $\epsilon>0$ such that for $r\ge4$, if $\|\mathcal{Q}_1\theta_0\|_r<\epsilon$, then there
exists  a global  solution $(\theta,L)\in H^r\left(\mathbb{T}[0,2\pi]\right)\times\mathbb{R}$, which satisfies  (A.1)-(A.2)
with initial condition (\ref{1.3}).  Further,
$\|\tilde\theta\|_r$, $\hat\theta (1)$ and $\hat\theta (-1)$ each
decay exponentially as $t \rightarrow \infty$, $|\hat\theta(0)|$
remains finite, while $L$ approaches $2\sqrt{\pi\mathcal{S}}$,
$\mathcal{S}$ being the  area of the bubble, which is invariant with
time. Thus a near circular bubble is asymptotically stable for
sufficiently small distortions in the
$H^r\big(\mathbb{T}[0,2\pi]\big) $ space.
\end{theorem}

In \S 2, we introduce a modified evolution system (B.1)-(B.4) with
initial condition (\ref{2.4}), which is shown to be equivalent to
(A.1)-(A.2) with initial condition (\ref{1.3}). We formulate a Galerkin approximation (\ref{2.12}) and
show how Theorem \ref{Thm1.7} follows from Theorem \ref{Thm2.15},
Lemma \ref{lem2.16} and Proposition \ref{prop2.18}.

In \S 3, we prove several preliminary Lemmas. In \S 4, we prove {\it
a priori} estimates on the growth of solutions to the approximate
initial value problem (\ref{2.12}). In \S 5, first we use  a priori
estimates to prove global existence and uniqueness of solutions to
the Galerkin approximation (\ref{2.12}), then show the same to be
true for (B.1)-(B.4) with initial condition (\ref{2.4}). Finally, we
also show that $\|\tilde\theta\|_r$, for the solution to (B.1)-(B.4)
with initial condition (\ref{2.4}) decays exponentially in time.

\section{Equivalent evolution equations}
In this section,  we derive an  equivalent system of the evolution
equations, which  will be analyzed later in this paper. Much of the
difficulty in this problem is to  control the energy appropriately.
We find that an equivalent system  provides exponentially decaying
energy estimates, unlike the original system which contains
neutrally stable modes corresponding to bubble translation
degeneracy.

\begin{definition}
\label{def2.1}
We introduce functions
\begin{equation*}
\omega_0 (\alpha) = \int_0^\alpha e^{i \alpha'} d\alpha' ~~~~~\\,~~~
\omega(\alpha)=
\int_0^{\alpha}e^{i\alpha'+i\hat\theta(1)e^{i\alpha'}+i\hat\theta(-1)
e^{-i\alpha'}+i\tilde\theta(\alpha')}d\alpha'.
\end{equation*}
\end{definition}

\begin{remark}
It is readily checked that
\begin{equation}
\label{Fomega01}
\Re\Big(\frac{\omega_{0,\alpha}}{\pi}PV\int_0^{2\pi}\frac{f(\alpha')}
{\omega_0(\alpha)-\omega_0(\alpha')}d\alpha'\Big)= \mathcal{H}(f).
\end{equation}
From expression (\ref{2.5.0}) and (\ref{2.5}), it is also
easily checked that
if
$f(\alpha)
=\sum_{k=-\infty}^{\infty}\hat{f}(k) e^{ik\alpha}$, then we have
\begin{equation}
\label{omega0f}
\mathcal{G} [\omega_0] \gamma
= i {\hat f} (0),
\end{equation}
which  from (\ref{2.5}) implies
\begin{equation}
\label{Fomega0} \mathcal{F} \left [ \omega_0 \right ] f =
\hat{f}(0).
\end{equation}
\end{remark}

We will show that evolution system (A.1)-(A.2) is equivalent to the following
evolution system for
$\big(\tilde{\theta}(\alpha,t),L(t),\hat{\theta}(0;t)\big)$ with
$\tilde\theta(\alpha,t)=\sum_{k\neq0,\pm1}\hat{\theta}(k;t)e^{ik\alpha}$,
where
$\theta(\alpha,t)=\hat{\theta}(0;t)+
\hat\theta(-1;t)e^{-i\alpha}+\hat\theta(1;t)e^{i\alpha}
+\tilde\theta(\alpha,t)$:
\begin{equation*}
\left\{\begin{aligned}\frac{\partial
\tilde{\theta}(\alpha,t)}{\partial t}&=\frac{2\pi}{L}
\mathcal{Q}_1\big(U_{\alpha}+T(1+\theta_{\alpha})\big),\\
\frac{d L(t)}{d t}&=-\int_0^{2\pi}\big(1+\theta_{\alpha}\big)
Ud\alpha,
\end{aligned}\right.\tag{B.1}
\end{equation*}
\begin{equation*}\frac{d \hat{\theta}(0;t)}{d t}=\frac{1}{L}\int_0^{2\pi}
T(1+\theta_{\alpha})d\alpha,\tag{B.2}
\end{equation*}
with $\gamma(\alpha,t)$, $T(\alpha,t)$, $\hat{\theta}(1;t)$ and
$\hat{\theta} (-1;t)$ determined\footnote{Since $\theta(\alpha,t)$
is real valued, note ${\hat \theta}^{\ast} (1; t) = \theta (-1,
t)$.} by
\begin{equation*}\left\{\begin{aligned}
\gamma(\alpha,t)&=-\frac{L} {\pi}A_{\mu} {\bf W} \cdot {\bf t}
+\frac{2\pi}{L}\sigma \theta_{\alpha \alpha},\\
T(\alpha,t)&=\int_0^{\alpha}
\big(1+\theta_{\alpha'}(\alpha')\big)U(\alpha')
d\alpha'-\frac{\alpha}{2\pi} \int_0^{2\pi}
\big(1+\theta_{\alpha}(\alpha)\big)U(\alpha) d\alpha,
\end{aligned}\right.\tag{B.3}
\end{equation*}
\begin{equation*}
\int_0^{2\pi}\exp\Big(i\alpha+i\big(\hat{\theta}(-1;t)e^{-i\alpha}
+\hat{\theta}(1;t)e^{i\alpha}+\tilde\theta(\alpha,t)\big)\Big)d\alpha=0,\tag{B.4}
\end{equation*}
where
 $U$ and ${\bf W} \cdot {\bf t}$ are given by
\begin{eqnarray}
\label{2.2}
  U
&=&\Re\Big(\frac{\omega_\alpha(\alpha)}{L}\mbox{PV}\int_0^{2\pi}\frac{\gamma(\alpha')}
 {\omega(\alpha)-\omega(\alpha')}d\alpha'\Big) \\
&=& \frac{\pi}{L} \mathcal{H} [\gamma] + \frac{\pi}{L} \Re \left (
\omega_\alpha \left [ \mathcal{H}, \frac{1}{\omega_\alpha} \right ]
\gamma + 2 i \omega_\alpha \mathcal{K} [\omega] \gamma \right
),\nonumber
\end{eqnarray}
\begin{equation}
\label{2.3}
{\bf W} \cdot {\bf t} =
\Re\Big(\frac{\omega_\alpha(\alpha)}{L i }\mbox{PV}\int_0^{2\pi}\frac{\gamma(\alpha')}
 {\omega(\alpha)-\omega(\alpha')}d\alpha'\Big)  =
\frac{\pi}{L} \mathcal{F} [\omega] \gamma.
\end{equation}
 \begin{remark}
 The formulae for $U$ in (\ref{2.2}) and ${\bf W}\cdot{\bf t}$ in (\ref{2.3})
 are equivalent to those in (\ref{Ueqn}) and (\ref{Wteqn}) since
 $\hat{\theta}(0;t)$ and $L$ cancel out.
 \end{remark}

The appropriate initial condition is
\begin{eqnarray}
\label{2.4}
\tilde\theta(\alpha,0)=\mathcal{Q}_1\theta_0,\,\,L(0)=2\pi,
\,\,\hat{\theta}(0;0)=\hat{\theta}_0(0).
\end{eqnarray}

Note that the first equation in (B.3) can be rewritten  as
\begin{eqnarray}
\label{2.6}
\big(I+A_{\mu}\mathcal{F}[\omega]\big)\gamma=\frac{2\pi}{L}\sigma
\theta_{\alpha\alpha}.
\end{eqnarray}
\begin{note}
\label{note2.3}
Later we shall see that if
$\tilde\theta\in H^1\big(\mathbb{T}[0,2\pi]\big)$ and $\|\tilde\theta\|_1$ is sufficiently small, then $I+A_{\mu}\mathcal{F}[\omega]$ is
invertible from $\{u\in
H^0\big(\mathbb{T}[0,2\pi]\big)|\hat{u}(0)=0\}$ to itself for any $A_\mu
\in [-1,1]$.
More general results are available \cite{AD1}, \cite{BMO} for non self-
intersecting interface; however, since we need the sharper estimates
for near circular interface in any case, we construct a direct proof rather
than rely on the more general theorems.
\end{note}

\begin{definition}
\label{def2.4}
Let $r\geq 3$.
 We define an open ball $\mathcal{B}$:
 $$\mathcal{B}=\left\{u\in
 \dot{H}^r| \|u\|_r<\epsilon\right\}.$$
 We also define  open balls:
  \begin{align*}
  \mathcal{O}&=\left\{(u,v,w)\in \dot{H}^r\times \mathbb{R}^2\big|
 u\in\mathcal{B},\, |v-2\pi|+|w|<
  1\right\},\\
  \mathcal{V}&=\left\{(u,v)\in \dot{H}^r\times \mathbb{R}
\big|
 u\in\mathcal{B},\, |v-2\pi|<
  1 \right\},\\
  \mathcal{U}&=\left\{(u,v)\in H^r\big(\mathbb{T}[0,2\pi]\big)\times \mathbb{R}\big|\mathcal{Q}_1u\in\mathcal{B},\,\|u\|_r+|v-2\pi|<
  1\right\}.
  \end{align*}
\end{definition}
\begin{remark}
We choose  $\epsilon>0$ is small enough for Lemma \ref{lem2.16} to
apply.
\end{remark}

For (B.4), we also have the following result:
\begin{prop}
\label{prop2.6} There exists $\epsilon_1>0$ so that  (B.4)
implicitly defines a unique $C^1$ function $g:\big\{u\in
\dot{H}^1|
\|u\|_1<\epsilon_1\big\}\rightarrow \mathbb{R}^2$ satisfying
$\big(\Re\hat\theta(1),\Im\hat{\theta}(1)\big)=g(\tilde\theta)$
and $g(0)=0$. Further, $g$ satisfies the following estimates  for
all $u, u_1,u_2\in\big\{u\in
\dot{H}^1|
\|u\|_1<\epsilon_1\big\}$:
\begin{eqnarray}
\label{2.7}
|g(u)|&\leq &\frac{1}{2}\|u\|_1,\\
|g(u_1)-g(u_2)|&\leq& \frac{1}{2}\|u_1-u_2\|_1.
\label{2.8}
\end{eqnarray}
\end{prop}
\begin{remark}
  Having determined $\gamma$,  $\hat\theta(1)$ and $\hat\theta(-1)$,
(\ref{2.2}) and the second equation in (B.3) determine $U$ and
 $T$ needed in (B.1) and (B.2).
\end{remark}
\begin{lemma}
\label{lem2.7}
   If $(\theta, L)\in C\big([0,S];\mathcal{U}\big)$ with $\theta$ real-valued is the solution of the evolution equations (A.1), where
    $\gamma$, $T$ and $U$ are determined by (A.2) and (\ref{1.2}) with  initial
condition (\ref{1.3}), then
$\big(\tilde{\theta}=\mathcal{Q}_1\theta,L,\hat{\theta}(0)\big)$
will satisfy the equations (B.1) and (B.2) where   $\gamma$, $T$,
$\hat{\theta}(\pm1)$ and $U$ are determined by  (B.3), (B.4) and
(\ref{2.2}) with initial condition (\ref{2.4}) for $t\in[0,S]$.

Conversely, if $\big(\tilde{\theta},L,\hat{\theta}(0)\big)\in
C\big([0,S];\mathcal{O}\big)$ is the solution of the system (B.1)
and (B.2) where   $\gamma$, $T$, $\hat{\theta}(\pm1)$ and $U$ are
determined by  (B.3), (B.4) and (2.4) with  initial condition
(\ref{2.4}), then   $\theta =\tilde{\theta}+\hat{\theta}(0)
+\hat{\theta}(1)e^{i\alpha}+\hat{\theta}(-1)e^{-i\alpha}$ is a
real-valued function and $(\theta,L)$  satisfies the system (A.1)
for $t \in [0, S]$ with initial condition (\ref{1.3}),   where
    $\gamma$, $T$ and $U$ are determined by (A.2) and (1.2).
\end{lemma}
\begin{proof}
Let $(\theta, L)\in C\big([0, S];\mathcal{U}\big)$ be the solution
of the evolution equations (A.1)  where
    $\gamma$, $T$ and $U$ are determined by (A.2) and (\ref{1.2})
with initial condition (\ref{1.3}). Then
 we define
$p(t)=\int_0^{2\pi} e^{i\alpha+i\theta(\alpha,t)}d\alpha$.  From the
consistency condition (\ref{1.3.1}),
$p(0)=0$. We consider
$$p'(t)=i\int_0^{2\pi}e^{i\alpha+i\theta(\alpha,t)}\theta_t
d\alpha. $$ Substituting for $\theta_t$ from (A.1), and using the
identity $(e^{i\alpha+i\theta})_{\alpha}=i(1+\theta_\alpha)
e^{i\alpha+i\theta}$, we have
$$ p'(t)=\frac{2\pi}{L}\int_0^{2\pi}\big[iU_\alpha e^{i\alpha+i\theta}+T(e^{i\alpha+i\theta})_\alpha\big]d\alpha.$$
We integrate the last  term by parts; we use (A.2) to substitute for
$T_\alpha$. There is no boundary term from integrating by parts
since $T$ and $e^{i\alpha+i\theta}$ are periodic. We have
\begin{eqnarray}
 p'(t)&=&\frac{2\pi}{L}\int_0^{2\pi}\big(iU_\alpha e^{i\alpha+i\theta}-(1+\theta_\alpha) U e^{i\alpha+i\theta}
 -\frac{1}{2\pi} L_t e^{i\alpha+i\theta}\big)d\alpha\nonumber.
 \end{eqnarray}
Since $iU_\alpha e^{i\alpha+i\theta}-(1+\theta_\alpha) U
e^{i\alpha+i\theta}=( i U e^{i\alpha+i\theta})_\alpha$, we have
$$p'=-\frac{L_t}{L}p.$$
Note that $(\theta,L)\in\mathcal{U}$ implies that $L>2\pi-1>0$.
Furthermore, $L_t$ is continuous in $[0, S]$ from (A.1). So $p(t)=0$
is the unique solution to the above ordinary differential equation
with $p(0)=0$ for $t\in[0,S]$. Hence
$$e^{i\hat{\theta}(0)}\int_0^{2\pi}\exp\Big(i\alpha+i\big(\hat{\theta}(-1;t)e^{-i\alpha}+\hat{\theta}
(1;t)e^{i\alpha}+\tilde\theta(\alpha,t)\big)\Big)d\alpha=0,$$ implying
 $$\int_0^{2\pi}\exp\Big(i\alpha+i\big(\hat{\theta}(-1;t)e^{-i\alpha}+\hat{\theta}
(1;t)e^{i\alpha}+\tilde\theta(\alpha,t)\big)\Big)d\alpha=0 \mbox{
 for  }t\in[0, S].$$ Thus
$\big(\tilde{\theta}=\mathcal{Q}_1\theta,
 L,\hat{\theta}(0)\big)$ satisfies
the  equations (B.1) and (B.2) where   $\gamma$, $T$,
$\hat{\theta}(\pm1)$ and $U$ are determined by (B.3), (B.4) and
(\ref{2.2}) with initial condition (\ref{2.4}) for $t\in[0,S]$.

Conversely, suppose that
$\big(\tilde{\theta},L,\hat{\theta}(0)\big)\in C\big([0, S];
\mathcal{O}\big)$ satisfies (B.1) and (B.2) with initial condition
(\ref{2.4}), where  $\gamma$, $T$,  $\hat{\theta}(\pm1)$ and $U$ are
determined by (B.3), (B.4) and (\ref{2.2}). Let
$\theta=\tilde{\theta}+\hat{\theta}(0)+\hat{\theta}(1)e^{i\alpha}+\hat{\theta}(-1)e^{-i\alpha}$.
We note from proposition \ref{prop2.6} that ${\hat \theta } (\pm1)$
scale as $\epsilon_1$ and hence is small. We note from (B.4) that
$$p(t)=
e^{i\hat\theta(0; t)}\int_0^{2\pi}\exp\Big(i\alpha+i\big(\hat{\theta}
(-1; t)
e^{-i\alpha}+\hat{\theta}(1; t)e^{i\alpha}+
\tilde\theta(\alpha; t)\big)\Big)d\alpha
=0.$$
It is convenient to define $\Gamma(\alpha,
t)=U_\alpha+T(1+\theta_\alpha)$. From $p^\prime (t)=0$, using
(B.1), we obtain
\begin{equation}
\label{2.9}
 0 =
\int_0^{2\pi}e^{i\alpha+i\theta}\Big(\big(\hat{\theta}_t(-1;t)
-\frac{2\pi}{L}\widehat{\Gamma}(-1;t)\big)
 e^{-i\alpha}+\big(\hat{\theta}_t(1;t)-\frac{2\pi}{L}
\widehat{\Gamma}(1;t)\big)e^{i\alpha}\Big)d\alpha.
\end{equation}
 Let
$e^{i\alpha+i\theta}=\sum_{k=-\infty}^\infty \hat{c}(k)
e^{ik\alpha}$.
Hence for sufficiently small ball size $\epsilon$ of $\mathcal{B}$,
using proposition \ref{prop2.6} and
Sobolev inequality $|. |_\infty < C \| . \|_1$,
$$ | \theta-\hat{\theta}(0) |_{\infty} = | {\tilde \theta} (\alpha, t)
+ {\hat \theta} (1; t) e^{i \alpha} + {\hat \theta} (-1; t)
|_{\infty}\leq C\|\tilde\theta\|_1 $$ is small, which clearly
ensures $|\hat{c}(1)|
> |\hat{c}(k)|$ for $k \ne 1$. Note further that (\ref{2.9}) implies
\begin{eqnarray*}
\big(\hat{\theta}_t(-1;t)-\frac{2\pi}{L}\widehat{\Gamma}(-1;t)\big)\hat{c}(1)
+\big(\hat{\theta}_t(1;t)-\frac{2\pi}{L}\widehat{\Gamma}(1;t)\big)\hat{c}(-1)=0.
\end{eqnarray*}
Since $\Gamma (\alpha, t)$ and $\tilde\theta$ are real valued,
$\hat{\theta}_t(-1;t)-\frac{2\pi}{L}\widehat{\Gamma}(-1;t)$ is the
complex conjugate of
$\hat{\theta}_t(1;t)-\frac{2\pi}{L}\widehat{\Gamma}(1;t)$. It is
clear that if $|a_1|\neq |a_2|$, then the only solution to $a_1
\eta+a_2 \eta^{\ast}=0$ is $\eta=0$. Hence
$$\hat{\theta}_t(-1;t)-\frac{2\pi}{L}\widehat{\Gamma}(-1;t)=
0\mbox{   and   }\hat{\theta}_t(1;t)-\frac{2\pi}{L}\widehat{\Gamma}(1;t)=0.$$
Hence,
$(\theta=\tilde{\theta}+\hat{\theta}(0)+\hat{\theta}(1)e^{i\alpha}+\hat{\theta}(-1)e^{-i\alpha},
L)$ will
satisfy  the system (A.1) where
$\gamma$, $T$ and $U$ are determined
by (A.2) and (\ref{1.2})
with  initial condition
(\ref{1.3}) for $t\in[0, S]$.
\end{proof}

We will henceforth discuss the global solutions of  the evolution
equations (B.1) where  $\gamma$, $T$,  $\hat{\theta}(\pm1)$ and $U$
are determined by (B.3), (B.4) and (\ref{2.2}) with initial
condition (\ref{2.4}).

\begin{definition}
\label{def2.8} Define $\hat\theta(1)=r_1+ir_2$. Then since $\theta$
is real valued,  $\hat\theta(-1)=r_1-ir_2$.
\end{definition}
\begin{remark}  (B.4)  becomes
\begin{equation}
\label{2.10}
\int_0^{2\pi}\exp\Big(i\alpha+i\big((r_1+ir_2)e^{i\alpha}+(r_1-ir_2)e^{-i\alpha}+\sum_{k=-\infty,\neq
0,\pm 1}^{\infty} \hat{\theta}(k) e^{ik\alpha}\big)\Big)d\alpha=0.
\end{equation}
\end{remark}

In order to prove Proposition \ref{prop2.6}, we need the following lemma:
\begin{lemma}
\label{lem2.9}
 Implicit function Theorem(\cite{Michael}):
Let $\mathcal{G}_1$, $\mathcal{G}_2$ and $\mathcal{G}_3$ be Banach
spaces and $F$ a mapping from an open subset of
$\mathcal{G}_1\times\mathcal{G}_2$ into $\mathcal{G}_3$. Let
$(u_0,v_0)$ be a point in $\mathcal{G}_1\times\mathcal{G}_2$
satisfying:
\begin{enumerate}
\item[(i)] $F(u_0,v_0)=0$; \item[(ii)] $F$ is continuously
differentiable at $(u_0,v_0)$; \item[(iii)] the partial
Fr$\acute{e}$chet derivative $D_vF(u_0,v_0)$ is invertible from
$\mathcal{G}_2$ to $\mathcal{G}_3$.
\end{enumerate}
Then, there are neighborhood $\mathcal{V}_1$ of $u_0$ in
$\mathcal{G}_1$ and neighborhood  $\mathcal{V}_2$ of $v_0$ in
$\mathcal{G}_2$  and a $C^1$ map $g:\mathcal{V}_1\rightarrow
\mathcal{V}_2$ so that $F\big(u,g(u)\big)=0$ for all $u\in
\mathcal{V}_1$ and for each $u\in \mathcal{V}_1$, $g(u)$ is the
unique point $v$ in $\mathcal{V}_2$ satisfying $F(u,v)=0$.
\end{lemma}

\begin{definition}
\label{def2.11}
 In the bubble context, we define
\begin{eqnarray}
F(u,v) &=&\int_0^{2\pi}\exp\Big(i\alpha+i\big(2(r_1
\cos{\alpha}-r_2\sin{\alpha})+u\big)\Big)d\alpha\nonumber
\end{eqnarray}
with $v=(r_1,r_2)$.
\end{definition}
\begin{remark}
 Note $F:\dot{H}^1\times\mathbb{R}^2\rightarrow\mathbb{C}$.
 \end{remark}

\noindent{\bf Proof of Proposition \ref{prop2.6}:}
 Let us show that the
Fr$\acute{e}$chet derivative of $F(u,v)$ with respect to $u$ exists
in $\dot{H}^1\times \mathbb{R}^2$. Since
\begin{eqnarray}
&&\Big|F(u+h,v)-F(u,v)-
\int_0^{2\pi}ih(\alpha)\exp\Big(i\alpha+i\big[2(r_1
\cos{\alpha}-r_2\sin{\alpha})+u(\alpha)\big]\Big)d\alpha\Big|\nonumber\\
&=&\Big|\int_0^{2\pi}\exp\Big(i\alpha+i\big[2(r_1
\cos{\alpha}-r_2\sin{\alpha})+u(\alpha)\big]\Big)\Big\{\exp\big[ih(\alpha)\big]-1-ih(\alpha)\Big\}d\alpha\Big|\nonumber\\
&\leq&2\pi \sum_{n=2}^{\infty}\frac{(\sum_{k=-\infty,\neq 0,\pm
1}^{\infty}|\hat{h}(k)|)^n}{n!}\leq 2\pi
\sum_{n=2}^{\infty}\frac{c^n\|h\|^n_1}{n!},\nonumber
\end{eqnarray}
 the Fr$\acute{e}$chet derivative of $F$ with respect to $u$ is
$$D_u F(u,v)h=\int_0^{2\pi}ih(\alpha)\exp\Big(i\alpha+i\big(2(r_1
\cos{\alpha}-r_2\sin{\alpha})+u(\alpha)\big)\Big)d\alpha,$$ for
$h\in \dot{H}^1$. It is clear that $D_u
F(u,v):\dot{H}^1\rightarrow\mathbb{C}$ is the bounded linear
operator  for all $(u,v)\in \dot{H}^1\times \mathbb{R}^2$.

 Similarly,
 $$D_v F(u,v)\delta v=
2i\int_0^{2\pi}\big(\delta r_1\cos\alpha-\delta
r_2\sin\alpha\big)\exp\Big(i\alpha+i\big(2(r_1
\cos{\alpha}-r_2\sin{\alpha})+u(\alpha)\big)\Big)d\alpha,$$ with
$\delta v=(\delta r_1, \delta r_2)\in \mathbb{R}^2$ is a bounded
linear operator for all
$(u,v)\in\dot{H}^1\times \mathbb{R}^2$, with
$$D_v F(0,0)\delta v=
2i\int_0^{2\pi}\big(\delta r_1\cos\alpha-\delta
r_2\sin\alpha\big)e^{i\alpha}d\alpha=2\pi(\delta r_2+i\delta r_1).$$
Clearly $D_v F(0,0)$ is invertible.
  So by the implicit function theorem (Lemma
\ref{lem2.9}), for $(u_0,v_0)=(0,0)$, there exist  neighborhood
$\mathcal{V}_1=\{u\in \dot{H}^1:\|u\|_1<2\epsilon_1\}$ of $0$ in
$\dot{H}^1$, and a neighborhood $\mathcal{V}_2 $ of $(0,0)$ in
$\mathbb{R}^2$, and a $C^1$ map $g: \mathcal{V}_1\rightarrow
\mathcal{V}_2$, so that $F(u,g(u))=0$ for all $u\in \mathcal{V}_1$.
We also have $\|D g(u)\|\leq \frac{1}{2}$ for $u\in\mathcal{V}_1$
since $D g(0)=0$. Hence we have
\begin{align*}
|g(u)|&\leq \Big|\int_0^1D g(tu)udt\Big|\leq \frac{1}{2}\|u\|_1,\\
|g(u_1)-g(u_2)|&\leq \Big|\int_0^1D
g\big(u_1+t(u_2-u_1)\big)(u_2-u_1)dt\Big|\leq \frac{1}{2}\|u_1-u_2\|_1
\end{align*} for all $u, u_1, u_2\in\{u\in
\dot{H}^1:\|u\|_1<\epsilon_1\}$.

\begin{corollary}\label{cortheta}
There exists sufficiently small $\epsilon_1>0$ so that for
$\theta\in H^{s+1}\left(\mathbb{T}[0,2\pi]\right)$ with $s\ge0$, if
$\| {\tilde \theta} \|_1 < \epsilon_1$,  then  $\theta$ satisfying
(B.4) implies $ \| \theta_\alpha \|_s \le 2 \| {\tilde
\theta}_\alpha \|_s $.
\end{corollary}
\begin{proof} We note from the relation between $\theta$ and ${\tilde \theta}$
that
$$ \| \theta_\alpha \|_s^2 = \sum_{k} |k|^{2s+2} |{\hat \theta} (k) |^2
= 2 |g ({\tilde \theta})|^2 + \| {\tilde \theta}_\alpha \|_s^2. $$
The rest follows from bounds on $g ({\tilde \theta}) $ in
Proposition \ref{prop2.6}.
\end{proof}

\subsection{Galerkin approximation.}

From the set of equations in (B.1)-(B.4), it is easily seen that
${\hat \theta} (0; t)$ does not effect the evolution of ${\tilde
\theta}$ and $L$, it is convenient to first determine solution
$({\tilde \theta}, L ) $; determination of ${\hat \theta} (0; t)$ is
then simply reduced to an integration of the  equation (B.2). It is
convenient to introduce a Galerkin approximations as described in
this section.

\begin{definition}
\label{def2.12}
We define a family of Galerkin projections
$\{P_n\}_{n=2}^{\infty}$, where
$$P_nu(\alpha)=\sum_{k=-n,k\neq 0,\pm 1}^n\hat{u}(k)e^{ik\alpha},\mbox{  for all }
u=\sum_{-\infty}^\infty\hat{u}(k)e^{ik\alpha}.$$
\end{definition}

We define the approximate solution $\tilde{\theta}_n(\alpha,t)$ of
order $n$ of the problem in the following way:
$$\tilde{\theta}_n(\alpha,t)=\sum_{k=-n,k\neq 0,\pm
1}^n\hat{\theta}_n(k;t)e^{ik\alpha}.$$

 The approximate equations are
\begin{equation*}\left\{\begin{aligned}
\frac{\partial \tilde\theta_n(\alpha,t)}{\partial t}&=\frac{2\pi}{L_n}P_n\Big( U_{n,\alpha}+T_n\big(1+\theta_{n,\alpha}\big)\Big),\\
\frac{dL_{n}(t)}{d t}&=-\int_0^{2\pi}\big(1+\theta_{n,\alpha}\big)
U_nd\alpha,
\end{aligned}\right.\tag{C.1}
\end{equation*}
with $\gamma_n$, $T_n$ and $\hat{\theta}_n(\pm1)$ (where
$\hat{\theta}_n^{\ast}(1)=\hat\theta_n(-1)$ because $\theta_n$ is real)
determined by
\begin{equation*}\left\{\begin{aligned}
\big(I+A_\mu\mathcal{F}&[\omega_n]\big)
\gamma_n(t)=\frac{2\pi}{L_n}\sigma\theta_{n,\alpha \alpha},\\
T_n (\alpha,t)&=\int_0^{\alpha}(1+\theta_{n,\alpha'})U_n (\alpha')d\alpha'
-\frac{\alpha}{2\pi} \int_0^{2\pi} (1+\theta_{n,\alpha'} ) U_n (\alpha') d\alpha',\\
\int_0^{2\pi}\exp\Big(&i\alpha+i\big(\hat{\theta}_n(-1;t)e^{-i\alpha}+\hat{\theta}_n(1;t)e^{i\alpha}+\tilde\theta_n(\alpha,t)\big)\Big)d\alpha=0,\end{aligned}\right.\tag{C.2}
\end{equation*}
where
\begin{eqnarray}
\theta_{n,\alpha}&=&\tilde{\theta}_{n,\alpha}-i\hat{\theta}_n(-1)e^{-i\alpha}+i\hat{\theta}_n(1)e^{i\alpha},\nonumber\\
\omega_n(\alpha)&=&\int_0^\alpha
e^{i\tau+i\hat\theta_n(-1)e^{-i\tau}+i\hat{\theta}_n(1)e^{i\tau}+i\tilde\theta_n(\tau)}d\tau,\nonumber\\
  U_n
 &=&\Re\Big(\frac{\omega_{n,\alpha}(\alpha)}{L_n}\mbox{PV}\int_0^{2\pi}\frac{\gamma_n(\alpha')}
 {\omega_n(\alpha)-\omega_n(\alpha')}d\alpha'\Big).
\label{2.11}
\end{eqnarray}
\subsection{Main results.}
Let $X_n=\big(\tilde\theta_n,\,L_n \big)$. The Galerkin approximate
equations  (C.1)-(C.2) reduce to an ODE in the Banach space
$\dot{H}^r\times\mathbb{R}$:
\begin{equation}
\label{2.12}
\frac{d X_n}{dt}=F_n(X_n),\quad
X_n(0)=(P_n \theta_0,\,2\pi),
\end{equation}
where $F_n(X_n)=\big(F_{n,1} (X_n),F_{n,2} (X_n) \big)$ are given by
\begin{eqnarray}
\label{2.12.0} F_{n,1} &=& \frac{2\pi}{L_n} P_n \left ( U_{n,\alpha}
+ T_n (1+\theta_{n,\alpha} ) \right ),
\\
F_{n,2} &=& -\int_0^{2\pi} (1+\theta_{n,\alpha} ) U_n (\alpha)
d\alpha. \label{2.12.1}
\end{eqnarray}
For the approximate equation (\ref{2.12}), we have the following results:
 \begin{prop}
\label{prop2.13} Assume $P_n \theta_0 \in \mathcal{B}$ for $r\ge3$.
For sufficiently small ball size $\epsilon$ of $\mathcal{B}$, there
exists the unique solution $X_n\in C^1\left ([0,S_n);\mathcal{V
}\right )$ to the ODE in Eq. (\ref{2.12}), where $S_n$ depends on
$n$, $r$ and $\epsilon$.
\end{prop}
\begin{remark}
We will prove this proposition in \S5 using Picard theorem (See for instance
Chapter 3 in \cite{MB}).
\end{remark}
\begin{prop}
\label{prop2.14} Assume $X_n = \left ( {\tilde \theta}_n , L_n
\right ) $ is a solution of the the initial value problem
(\ref{2.12}). Then there exists $\epsilon>0$ such that if $\|P_n
\theta_0\|_r <\epsilon$ for $r \ge 3$, then
$$ \|\tilde\theta_n(\cdot,t)\|_r\leq
\|P_n \theta_0(\cdot)\|_re^{-\frac{1}{36}\sigma t}, ~~~~ |L_n^3 - 8
\pi^3 | \le C \epsilon \left ( 1 - e^{-\frac{1}{18} \sigma t} \right
), $$ with a constant $C$ independent of $n$ for any time $t \ge 0$
where the solution exists.
\end{prop}
\begin{remark} We will prove the priori estimate in \S4.
\end{remark}

\begin{theorem}
\label{Thm2.15}
Given the initial condition $X_n(0)
 \in \mathcal{V}$, for any $n\geq 2$ and $r\ge3$. For sufficiently small $\epsilon$,
there exists for all time  a unique solution $X_n(t)\in
C^1([0,\infty);\mathcal{V})$  to the approximate equation
(\ref{2.12}).
 \end{theorem}
 \begin{proof}
 Proposition \ref{prop2.13} shows the existence and uniqueness of
 solutions $X_n$ locally in time. Then by  continuation of an
 autonomous ODE
 on a Banach space (see \cite{MB} in Chapter 3),
we know that  the unique solution $X_n \in C^1([0, T_0 ]
 );\mathcal{V})$  either exists globally in
  time or $T_0<\infty$ and $X_n(t)$ leaves the open set $\mathcal{V}$
as $t\nearrow T_0$.
  Suppose $T_0<\infty$. Combining
Propositions \ref{prop2.14} and \ref{prop4.2},
we know that solution remains in
the open set $\mathcal{V}$ as $t\nearrow T_0$. Hence it
shows that  solution  to Eq. (\ref{2.12}) exists globally in time.
\end{proof}
From the solutions to the  approximate equation (\ref{2.12}),  we
will deduce the existence and uniqueness  of solutions to the
evolution system (B.1), (B.3) and (B.4) globally in time (Theorem
\ref{Thm1.7}) using the following lemma and proposition:
\begin{lemma}
\label{lem2.16} For $r\geq 4$,   there exists sufficiently small
$\epsilon>0$ such that for any $S > 0$,  solutions
$X_n=(\Tilde\theta_n,L_n )$ of the approximate equation (\ref{2.12})
for different $n$ form a Cauchy sequence in $C\left
([0,S];\dot{H}^1\times\mathbb{R} \right )$. As $n\rightarrow\infty$,
the limit $X=\big(\tilde\theta,L \big)\in
C\big([0,S];\mathcal{V}\big)\cap
C^1\left([0,S];\dot{H}^{r-3}\times\mathbb{R}\right)$ and is the
unique classical solution to (B.1), (B.3) and (B.4) satisfying
initial condition (\ref{2.4}).
\end{lemma}
\begin{remark} The proof is given in \S 5.
\end{remark}
\begin{definition}
\label{2.17} We define the area of bubble by $\mathcal{S}(t)$. Then
\begin{eqnarray}
\label{5.5} \mathcal{S}(t)=\frac{1}{2}\Im\int_0^{2\pi}z_\alpha
z^{\ast}d\alpha.
\end{eqnarray}
\end{definition}
\begin{prop}
\label{prop2.18} Let $(\Tilde\theta, L ) \in C\left
([0,\infty);\mathcal{V}\right )$ be a solution to  the  system
(B.1), (B.3) and (B.4) with  initial condition (\ref{2.4}) for
$r\geq 4$. If $\mathcal{Q}_1\theta_0\in\mathcal{B}$, then  the area
$\mathcal{S}$ is invariant with time and for all $t\ge0$, we have
$$ \|\tilde\theta(\cdot,t)\|_r\leq \| \mathcal{Q}_1 \theta_0(\cdot)\|_r
e^{-\frac{1}{36}\sigma t},$$
$$|\hat\theta(1;t)|=|\hat\theta(-1;t)|\leq \frac12\| \mathcal{Q}_1 \theta_0(\cdot)\|_r
e^{-\frac{1}{36}\sigma t},$$
\begin{equation*}
|L(t)-2\sqrt{\pi\mathcal{S}}|\leq
C\|\mathcal{Q}_1\theta_0\|_re^{-\frac{1}{36}\sigma t},
\end{equation*}
\begin{equation*}
|\hat{\theta}(0;t)-\hat{\theta}_0(0)|\leq
C\|\mathcal{Q}_1\theta_0\|_r,
\end{equation*}
where $C$ depends on $\mathcal{S}$ and the diameter of
$\mathcal{B}$.
\end{prop}
\begin{remark}
We will prove Lemma \ref{lem2.16} and Proposition  \ref{prop2.18} in
\S5. Further, the result above together with Proposition
\ref{prop2.6} shows that $\theta(\alpha,t)-\hat{\theta}(0;t)$ goes
to 0 exponentially as $t\rightarrow \infty$.
\end{remark}

\noindent{\bf Proof of Theorem \ref{Thm1.7}:} This immediately
follows from Lemma \ref{lem2.16} and Proposition \ref{prop2.18}
since Lemma \ref{lem2.7} gives equivalence between (A.1)-(A.2) and
(B.1)-(B.4).

\section{Preliminary Lemmas} We will need to use a variety of routine
estimates for integral operators and other functions in terms of
$\tilde\theta$ and $\hat{\theta}(0)$.  Recall tangent angle of the
curve is
$\alpha+\theta(\alpha)=\alpha+\tilde\theta(\alpha)+\hat{\theta}(0)+\hat{\theta}(-1)e^{-i\alpha}+\hat{\theta}(1)e^{i\alpha}$,
where $\hat\theta(1)$ and $\hat\theta(-1)$ are determined by
$g(\tilde\theta)$.

The next lemma gives a bound for $\omega_\alpha$ in terms of ${\tilde \theta}$.
 \begin{lemma}
\label{lem3.1} Assume $ \| \tilde \theta \|_1 < \epsilon_1$ where
$\epsilon_1$ is small enough for Corollary \ref{cortheta} to apply.

If $\omega$ determined by ${\tilde \theta} \in \dot{H}^s$, then for
$s \ge 1$,
 \begin{equation}
 \label{1ststate}
\|\omega_\alpha\|_s\leq
C_1(\|\tilde\theta\|_s+1)\exp\left(C_2\|\tilde\theta\|_{s-1}\right),~~~~~~~~~
\Big\|\frac{1}{\omega_\alpha}\Big\|_s\leq
C_1(\|\tilde\theta\|_s+1)\exp\left(C_2\|\tilde\theta\|_{s-1}\right),
\end{equation}
where constants $C_1$ and $C_2$, depend only on $s$, and particularly for $s=1$,
$C_2=0$.

Further, if $\omega^{(1)}, \omega^{(2)}$ correspond respectively to
${\tilde \theta}^{(1)}, {\tilde \theta}^{(2)} \in \dot{H}^{s} $,
where $\| {\tilde \theta}^{(1)} \|_1 $, $\| {\tilde \theta}^{(2)}
\|_1 < \epsilon_1$, then for $s \ge 1$,
\begin{eqnarray}
\label{3rdstate}
 \|\omega_{\alpha}^{(1)} - \omega_\alpha^{(2)} \|_{s} \leq
 C_1 \| {\tilde \theta}^1 - {\tilde \theta}^2  \|_s
\exp{\left [ C_2 \left ( \|{\tilde \theta}^1 \|_{s} + \| {\tilde
\theta}^2 \|_s
\right )\right]},\\
\left\|\frac{1}{\omega_{\alpha}^{(1)}} -
\frac{1}{\omega_\alpha^{(2)}}\right \|_{s} \leq
 C_1 \| {\tilde \theta}^{(1)} - {\tilde \theta}^{(2)}  \|_s
\exp{\left [C_2 \left ( \|{\tilde \theta}^{(1)} \|_{s} + \| {\tilde
\theta}^{(2)} \|_s \right ) \right ]},
\label{4thstate}
\end{eqnarray}
while for $s\ge2$,
\begin{eqnarray}
\label{5thstate} \|\omega_\alpha^{(1)} - \omega_\alpha^{(2)}\|_s
&\leq& C_1\left(\|\tilde\theta^{(1)}-\tilde\theta^{(2)}\|_s
+\|\tilde\theta^{(2)}\|_s\|\tilde\theta^{(1)}-\tilde\theta^{(2)}\|_{s-1}\right)
\\
&\times&
\exp\left[C_2\big(\|\tilde\theta^{(1)}\|_{s-1}
+\|\tilde\theta^{(2)}\|_{s-1}\big)\right],\nonumber\\
\label{6thstate} \left\|\frac{1}{\omega_\alpha^{(1)}} -\frac{1}{
\omega_\alpha^{(2)}}\right\|_s &\leq&
C_1\left(\|\tilde\theta^{(1)}-\tilde\theta^{(2)}\|_s
+\|\tilde\theta^{(2)}\|_s\|\tilde\theta^{(1)}-\tilde\theta^{(2)}\|_{s-1}\right)
\\
&\times&
\exp\left[C_2\big(\|\tilde\theta^{(1)}\|_{s-1}
+\|\tilde\theta^{(2)}\|_{s-1}\big)\right],
\nonumber
\end{eqnarray}
where the constants $C_1$ and $C_2$ depend only on $s$.
 \end{lemma}
 \begin{proof}
For the formula $\omega_\alpha=e^{i\alpha+i\theta-i {\hat \theta} (0)}$, it is
easy to obtain
 $$\|\omega_{\alpha}\|_0\leq C.$$
 Let us consider for $0<k\leq s$.
The chain rule gives
\begin{eqnarray}
D^k \omega_\alpha= \sum_{\beta_1+\cdots+\beta_{\mu}=k,\beta_i\geq 1}
C_\beta D ^{\beta_1}(\alpha+\theta)\cdots D
^{\beta_\mu}(\alpha+\theta) \omega_\alpha. \nonumber
\end{eqnarray}
So by Sobolev embedding Theorem, $|f |_{\infty}\leq C \|f\|_1$, we
have
\begin{equation}
\label{boundomega} \|D^k \omega_{\alpha}\|_0\leq C
\|1+\theta_{\alpha}\|_{k-1}(1+\|\theta_{\alpha}\|_{k-1}+\cdots+\|\theta_{\alpha}\|_{k-1}^{k-1})\leq
C_1 \exp{(C_2\|\theta_{\alpha}\|_{k-1})},
\end{equation}
where the constants, $C_1$ and $C_2$, depend only on $s$.

For $s=1$, we have
$$\|D\omega_\alpha\|_0=\|1+\theta_\alpha\|_0\leq
C(1+\|\tilde\theta\|_1).$$ For $s\ge2$, we note
$$D^s\omega_\alpha=D^{s-1}\big[i(1+\theta_\alpha) \omega_\alpha\big].$$ Hence, by noting Banach algebra
property (see Note \ref{note1.1}), Corollary \ref{cortheta} and
(\ref{boundomega}), we get
$$\|D^s\omega_\alpha\|_0\leq \Big\|
i(1+\theta_\alpha) \omega_\alpha\Big\|_{s-1}\leq
C_1(\|\tilde\theta\|_s+1)\exp\left(C_2\|\tilde\theta\|_{s-1}\right),$$where
the constants, $C_1$ and $C_2$, depend only on $s$. Since
$\frac{1}{\omega_\alpha}= e^{-i \alpha - i \theta (\alpha) +i {\hat
\theta} (0)} $, the preceding arguments are clearly applied to
$\frac{1}{\omega_\alpha}$ as well and (\ref{1ststate}) follows from
Corollary \ref{cortheta} for a modified constant $C_2$.

To prove (\ref{3rdstate}), we note that
$$ \omega_\alpha^{(1)} - \omega_\alpha^{(2)} =
\left [ e^{ i \big(\theta^{(1)}-\hat{\theta}^{(1)}(0) -
\theta^{(2)}+\hat{\theta}^{(2)}(0) \big)} - 1 \right ] e^{i\alpha+i
\theta^{(2)} - i {\hat \theta}^{(2)} (0) }
$$
From series representation of the exponential and application of
Banach algebra
property of $\| . \|_s$ norm to each term in the series, we deduce
\begin{multline*} \Big\| e^{ i
\big(\theta^{(1)}-\hat{\theta}^{(1)}(0) -
\theta^{(2)}+\hat{\theta}^{(2)}(0) \big)} - 1  \Big\|_s \\
\le C_1 \|  \theta^{(1)}-\hat{\theta}^{(1)}(0) -
\theta^{(2)}+\hat{\theta}^{(2)}(0)  \|_s \exp \left ( C_2 \|
\theta^{(1)}-\hat{\theta}^{(1)}(0) -
\theta^{(2)}+\hat{\theta}^{(2)}(0)  \Big\|_s \right
),\end{multline*} where the constants, $C_1$ and $C_2$, depend only
on $s$. Using Banach algebra properties and Corollary
\ref{cortheta}, (\ref{3rdstate}) follows. Almost identical arguments
are applied to prove (\ref{4thstate}).

Further, if $s\ge2$  we have
\begin{multline*}
\left\|D^{s}(\omega_\alpha^{(1)} - \omega_\alpha^{(2)})\right\|_0
=\left\|D^{s-1}\left[i(1+\theta_\alpha^{(1)})\omega_\alpha^{(1)}
-i(1+\theta_\alpha^{(2)})\omega_\alpha^{(2)}\right]
\right\|_0\\
\leq
C_1\left(\|\tilde\theta^{(1)}-\tilde\theta^{(2)}\|_s+\|\tilde\theta^{(2)}\|_s\|\tilde\theta^{(1)}-\tilde\theta^{(2)}\|_{s-1}\right)
\exp\left(C_2\|\tilde\theta^{(1)}\|_{s-1}+C_2\|\tilde\theta^{(2)}\|_{s-1}\right),
\end{multline*}
where the constants, $C_1$ and $C_2$, depend only on $s$. So
(\ref{5thstate}) follows. Almost identical arguments are applied for
(\ref{6thstate}).
\end{proof}

In simplifying our integral operators, we find
divided
differences to be very useful.
\begin{definition}
\label{def3.2}
We define
divided differences $q_1$ and $q_2$ as follows:
\begin{displaymath}
q_1[\omega](\alpha,\alpha')=\frac{\omega(\alpha)-\omega(\alpha')}{
\alpha-\alpha'}=\int_0^1
\omega_\alpha(t\alpha+(1-t)\alpha')dt,
\end{displaymath}
\begin{eqnarray}
q_2[\omega](\alpha,\alpha')&=&\frac{\omega(\alpha)-\omega(\alpha')-
\omega_{\alpha}(\alpha)(\alpha-\alpha')}{(\alpha-\alpha')^2}
=\int_0^1(t-1)\omega_{\alpha\alpha}((1-t)\alpha+t\alpha')dt.\nonumber
\end{eqnarray}
\end{definition}

\begin{prop}
\label{prop3.4} There exists $\epsilon_1 > 0$
so that
$\|\tilde\theta\|_1\leq \epsilon_1$ implies
\begin{eqnarray}
\label{3.1}
|q _1 [\omega] (\alpha, \alpha') | \ge \frac{1}{8} ~~,~~{\rm for} ~~
0 < |\alpha -\alpha'| \le \pi
\end{eqnarray}
\end{prop}
\begin{proof}
We note that
$$ q_1 [\omega] (\alpha, \alpha') =
\frac{\int_{\alpha'}^{\alpha}
e^{i\tau+i\tilde\theta(\tau)+i\hat{\theta}(1)e^{i\tau}+i\hat{\theta}(-1)e^{-i\tau}}d\tau
} {\alpha-\alpha'}.
$$ Further,
\begin{eqnarray*}
&&\Big|\frac{\int_{\alpha'}^{\alpha}e^{i\tau+i\tilde\theta(\tau)+i\hat{\theta}(1)e^{i\tau}+i\hat{\theta}(-1)e^{-i\tau}}d\tau}{\alpha-\alpha'}
-\frac{\int_{\alpha'}^{\alpha}e^{i\tau}d\tau}{\alpha-\alpha'}\Big|\nonumber\\
&=&
\Big|\frac{\int_{\alpha'}^{\alpha}e^{i\tau}(e^{i\tilde\theta(\tau)+i\hat{\theta}(1)e^{i\tau}
+i\hat{\theta}(-1)e^{-i\tau}}-1)d\tau}{\alpha-\alpha'}\Big|\\
&\leq&
2\sqrt{2}\max_{\tau\in[0,2\pi]}\big|\tilde\theta(\tau)+\hat{\theta}(1)e^{i\tau}+\hat{\theta}(-1)e^{-i\tau}\big|.\nonumber
\end{eqnarray*}
This bound is a consequence of the inequality
\begin{eqnarray}
|e^{i\zeta}-e^{i\zeta'}|\leq \sqrt{2}|\zeta-\zeta'|,\mbox{ for all
} \zeta,\,\zeta'\mbox{ in }\mathbb{R}.\nonumber
\end{eqnarray}
We choose $\epsilon_1>0$ small enough so that
Proposition \ref{prop2.6} holds and
from Sobolev
embedding theorem,
$$2\sqrt{2}\max_{\tau\in[0,2\pi]}\big|\tilde\theta(\tau)+\hat{\theta}(1)e^{i\tau}+\hat{\theta}(-1)e^{-i\tau}\big|\leq
c\|\tilde\theta\|_1\leq \frac{1}{8},$$ where $c$ is some constant.

 It is easy to see that
\begin{eqnarray}
\Big|\frac{\int_{\alpha'}^{\alpha}e^{i\tau}d\tau}{\alpha-\alpha'}\Big|\geq
\frac{1}{4}, \mbox{ for }0<|\alpha-\alpha'|\leq\pi.\nonumber
\end{eqnarray}
Thus, if $\|\tilde\theta\|_1\leq \epsilon_1$, we have
\begin{eqnarray}
\Big|\frac{\int_{\alpha'}^{\alpha}
e^{i\tau+i\tilde\theta(\tau)+i\hat{\theta}(1)e^{i\tau}+i\hat{\theta}(-1)e^{-i\tau}}d\tau}{\alpha-\alpha'}\Big|\geq
\frac{1}{8}.\nonumber
\end{eqnarray}
 \end{proof}

\begin{lemma}{(See \cite{AD1} or appendix for proof)}
\label{lem3.3} Let $\omega_\alpha \in H^{k}
\big(\mathbb{T}[0,2\pi]\big)$ for $k \ge 0$. Then $D_\alpha^k q_1,
D_{\alpha'}^k q_1 \in H^0 [a, a+2 \pi]$ in both variables $\alpha$
or $\alpha'$ and satisfy bounds
$$ \| D_\alpha^k q_1[\omega] \|_{0} \leq C \|\omega_\alpha\|_{k} ~~~~,~~~~~
\| D_{\alpha'}^k q_1[\omega] \|_{0} \leq C \|\omega_\alpha\|_{k}$$
with $C$ only depending on $k$ (in particular independent of $a$).
Further if $ \omega_{\alpha \alpha} \in H^{k} \big ( \mathbb{T} [0,
2 \pi] \big ) $ for $k \ge 0$, then $D_\alpha^k q_2, D_{\alpha'}^k
q_2 \in H^0 [a, a+2 \pi]$ in both variables $\alpha$ and $\alpha'$
and satisfy
$$ \| D_\alpha^k q_2 [\omega] \|_{0} \leq C \|\omega_{\alpha\alpha} \|_{k}
~~~~,~~~~~ \| D_{\alpha'}^k q_2
[\omega] \|_{0} \leq C
\|\omega_{\alpha \alpha} \|_{k}$$ with $C$ only depending on $k$.
\end{lemma}

\begin{lemma}
\label{lem3.5}
Let  $\omega^{(1)}, \omega^{(2)} \in H^{k+1}
\big(\mathbb{T}[0,2\pi]\big)$ for integer
$k\geq 0$.
Suppose
$$ |q_1[\omega^{(1)}](\alpha,\alpha')|\geq \frac{1}{8}, \mbox{ for }
0<|\alpha-\alpha'|\leq \pi.$$
Then
\begin{eqnarray}
\Big(\int_{\alpha-\pi}^{\alpha+\pi}\Big|D_\alpha^k
\frac{q_2[\omega^{(2)}](\alpha,\alpha')}{q_1[\omega^{(1)}](\alpha,\alpha')}
\Big|^2d\alpha'\Big)^{\frac{1}{2}}\leq
C_1\|\omega^{(2)}_\alpha\|_{k+1}
\exp(C_2\|\omega^{(1)}_\alpha\|_{k}), \label{5.3.2}
\end{eqnarray}
\begin{eqnarray*}
\Big(\int_{\alpha-\pi}^{\alpha+\pi}\Big|D_\alpha^k
\frac{q_2[\omega^{(2)}](\alpha',\alpha)}{q_1[\omega^{(1)}](\alpha',\alpha)}
\Big|^2d\alpha'\Big)^{\frac{1}{2}}\leq
C_1\|\omega^{(2)}_\alpha\|_{k+1}
\exp(C_2\|\omega^{(1)}_\alpha\|_{k}),
\end{eqnarray*}
where $C_1$ and $C_2$ depend on $k$ alone, but not on $\alpha$.
\end{lemma}

\begin{proof}
Clearly for $k=0$, (\ref{5.3.2}) holds. Consider $k\ge1$. It is easy
to have
$$ D_\alpha^k \frac{q_2}{q_1} = \sum_{j=0}^k C_{k,j} D_\alpha^{k-j} q_2
D^{j}_\alpha \frac{1}{q_1}. $$ We have from Lemma \ref{lem3.3}, for
$0 \le j \le k$,
$$ \| D_\alpha^{k-j} q_2 [\omega^{(2)} ](\alpha, \alpha') \|_0
\le C_1 \| \omega^{(2)}_{\alpha \alpha} \|_{k-j}  $$ and
$$ \| D_\alpha^{j} q_1 [\omega^{(1)} ](\alpha, \alpha') \|_0
\le C_1 \| \omega^{(1)}_{\alpha} \|_{j}.  $$ Further, since $q_1$ is
bounded below by $\frac{1}{8}$, it follows that for $0 \le j \le k$,
Hence
$$\| D^j_\alpha \frac{1}{q_1} \|_0 \le C \exp \left [ c_2
\sum_{m=1}^j \| D^m_\alpha q_1 \|_0 \right ] \le C_1 \exp \left [
C_2 \| \omega^{(1)}_\alpha \|_j \right ].$$
\begin{equation*}\left\|D_\alpha^k \frac{q_2}{q_1}\right\|_0\leq \sum_{j=0}^k
C_{k,j}\left\| D_\alpha^{k-j} q_2 D^{j}_\alpha
\frac{1}{q_1}\right\|_0\leq C_1\|\omega^{(2)}_\alpha\|_{k+1}
\exp(C_2\|\omega^{(1)}_\alpha\|_{k}),
\end{equation*}
since $\| \frac{1}{q_1} D_\alpha^k q_2  \|_0 \le
|\frac{1}{q_1}|_\infty \| D_\alpha^k q_2 \|_0 $ and for $1 \le j \le
k$,
$$\| D_\alpha^{k-j} q_2 ~~D_\alpha^{j} \frac{1}{q_1} \|_0 \le
| D_\alpha^{k-j} q_2 |_\infty
\| D_\alpha^j \frac{1}{q_1} \|_0
\le c \| q_2 \|_{k-j+1} \| D_\alpha^j \frac{1}{q_1}  \|_0 .$$
The second part follows in a very similar manner
since Lemma \ref{lem3.3} can be applied by switching variables
$\alpha'$ and $\alpha$ in the expression. We note that
Lemma \ref{lem3.3}
gives the same
$H^0 [a, a +2 \pi]$ estimates for
derivatives of $q_1$ and $q_2$ with respect to $\alpha$ or
$\alpha'$, independent of $a$.
\end{proof}

\begin{definition}
We write the cotangent as a function which is analytic at the
 origin plus a singular part:
 \begin{displaymath}
 \cot(\beta)=\frac{1}{\beta}+l(\beta).
 \end{displaymath}
\end{definition}

\begin{lemma}{(See \cite{AD1} or appendix for proof)}
\label{lem3.7}
Let $s\geq2$ and  $\omega\in H^s\big(\mathbb{T}[0,2\pi]\big)$ with
corresponding $\| {\tilde \theta} \|_1$ sufficiently small to ensure
$| q_1 [\omega] (\alpha, \alpha') | \ge \frac{1}{8}$. Then
$\mathcal{K}[\omega]:H^0\big(\mathbb{T}[0,2\pi]\big) \rightarrow
H^{s-2}\big(\mathbb{T}[0,2\pi]\big)$, and in particular, there are
positive constants $C_1$ and $C_2$ depending on $s$
such that
\begin{eqnarray}
\label{3.2}
\|\mathcal{K}[\omega]f\|_{s-2} \leq C_1 \|f\|_0
\exp{(C_2\|\omega_\alpha\|_{s-1})}.
\end{eqnarray}
Further, $\mathcal{K}[\omega]: H^1\big(\mathbb{T}[0,2\pi]\big)
\rightarrow H^{s-1}\big(\mathbb{T}[0,2\pi]\big)$, and
\begin{eqnarray}
\label{3.3}
\|\mathcal{K}[\omega]f\|_{s-1} \leq C_1\|f\|_1
\exp{(C_2\|\omega_\alpha\|_{s-1})}.
\end{eqnarray}
\end{lemma}

\begin{lemma}
\label{lem3.8} If $f \in H^1\left( \mathbb{T} [0, 2\pi] \right) $,
$\omega^{(1)}$ and $\omega^{(2)}$ correspond to
$\tilde{\theta}^{(1)}$ and $\tilde{\theta}^{(2)}$, each in $
\dot{H}^{1}$, respectively with $\|\tilde\theta^{(1)}\|_1$,
$\|\tilde\theta^{(2)}\|_1<\epsilon_1$, then for sufficient small
$\epsilon_1$,
\begin{eqnarray}
\|\mathcal{K}[\omega^{(1)}]f-\mathcal{K}[\omega^{(2)}]f\|_0 &\leq&
C_1
\|f\|_0\exp\left(C_2(\|\tilde\theta^{(1)}\|_1+\|\tilde\theta^{(2)}\|_1\right)
\|\tilde\theta^{(1)}-\tilde\theta^{(2)}\|_1.\nonumber
\end{eqnarray}
Suppose  ${\tilde \theta}^1, {\tilde \theta}^2 \in \dot{H}^s$. Then
for $s \ge 1$,
\begin{eqnarray*}
&&\|\mathcal{K}[\omega^{(1)}]f-\mathcal{K}[\omega^{(2)}]f\|_s \\
&\leq& C_1\exp\Big({C_2}\big(\|\tilde\theta^{(1)}\|_{s}+
\|\tilde\theta^{(2)}\|_{s}\big)\Big)
\|\tilde\theta^{(1)}-\tilde\theta^{(2)}\|_s\|f\|_1,
\end{eqnarray*}
while for
for $s\ge 3$,
\begin{eqnarray*}
&&\|\mathcal{K}[\omega^{(1)}]f-\mathcal{K}[\omega^{(2)}]f\|_s \\
&\leq& C_1\exp\Big({C_2}\big(\|\tilde\theta^{(1)}\|_{s-1}+
\|\tilde\theta^{(2)}\|_{s-1}\big)\Big)\Big(\big(\|\tilde\theta^{(1)}\|_s+\|\tilde\theta^{(2)}\|_s\big)\|\tilde\theta^{(1)}-\tilde\theta^{(2)}\|_{s-1}\\
&&\hspace{6cm}+
\|\tilde\theta^{(1)}-\tilde\theta^{(2)}\|_{s}\Big)\|f\|_1,
\end{eqnarray*}
where constants $C_1$ and $C_2$ depend on $s$ only.
\end{lemma}
\begin{proof}
We note that
\begin{multline*}
\mathcal{K} [\omega^{(1)}] f - \mathcal{K} [\omega^{(2)} ] f =
-\frac{1}{2 \pi i} \int_{\alpha-\pi}^{\alpha+ \pi}
\frac{f(\alpha')}{\omega^{(1)}_\alpha(\alpha')}  \left ( \frac{q_2
[\omega^{(1)}] (\alpha', \alpha) }{q_1 [\omega^{(1)} ] (\alpha',
\alpha) } -\frac{q_2 [\omega^{(2)}] (\alpha', \alpha) }{q_1
[\omega^{(2)} ] (\alpha', \alpha) }  \right )
d\alpha' \\
-\frac{1}{2 \pi i} \int_{\alpha-\pi}^{\alpha+ \pi}
f(\alpha')\left(\frac{1}{\omega^{(1)}_\alpha(\alpha')}-\frac{1}{\omega^{(2)}_\alpha(\alpha')}\right)
 \frac{q_2 [\omega^{(2)}] (\alpha', \alpha) }{q_1 [\omega^{(2)} ]
(\alpha', \alpha) }
d\alpha' \\
-\frac{1}{2\pi i}\int_{\alpha-\pi}^{\alpha+\pi} f(\alpha')\left(
\frac{1}{2\omega^{(1)}_{\alpha}(\alpha')}-
\frac{1}{2\omega^{(2)}_{\alpha}(\alpha')}\right)
l\big(\frac{1}{2}(\alpha-\alpha')\big)d\alpha' . \end{multline*} We
also have \begin{multline*}
  \frac{q_2
[\omega^{(1)}] (\alpha', \alpha) }{q_1 [\omega^{(1)} ] (\alpha',
\alpha) } -\frac{q_2 [\omega^{(2)}] (\alpha', \alpha) }{q_1
[\omega^{(2)} ] (\alpha', \alpha) } =\frac{q_2
[\omega^{(1)}-\omega^{(2)} ] (\alpha',\alpha) }{ q_1 [\omega^{(1)}
](\alpha', \alpha) } \\-\frac{q_2 [\omega^{(2)} (\alpha', \alpha)
q_1 [\omega^{(1)}-\omega^{(2)} ] (\alpha',\alpha) }{ q_1
[\omega^{(2)} ](\alpha', \alpha) q_1 [\omega^{(1)} ](\alpha',
\alpha) }.
\end{multline*}
Therefore, using Sobolev inequality $|. |_\infty \le C \| . \|_1$, we obtain
\begin{multline*}
\| \mathcal{K} [\omega^{(1)}] f - \mathcal{K} [\omega^{(2)} ] f \|_0
\le C_1 \| f \|_0\left\|\frac{1}{\omega^{(1)}_\alpha}\right\|_1
\left ( \| q_2 [\omega^{(1)} -\omega^{(2)} ]  \|_0 \right . \\
\left . +\| q_1 [\omega^{(1)} - \omega^{(2)}] \|_1 \| q_2
[\omega^{(2)}  \|_0 \right )+C_2\| f
\|_0\left\|\frac{1}{\omega^{(1)}_\alpha}-\frac{1}{\omega^{(2)}_\alpha}\right\|_1\left(\|q_2
[\omega^{(2)}] \|_0+1\right).
\end{multline*}
The first statement follows easily from Lemmas \ref{lem3.1},
\ref{lem3.3} and \ref{lem3.5}. Further, using one integration by
parts, the $s$th derivative of $\mathcal{K}[\omega] f$ is
\begin{eqnarray}
D^{s}_\alpha \mathcal{K}[\omega] f (\alpha) &=&D^{s-1}_\alpha
\frac{1}{2\pi i}\int_{\alpha-\pi}^{\alpha+\pi}
D_{\alpha'}\Big(\frac{ f(\alpha')}{\omega_\alpha(\alpha')}\Big)
\Big{[}\frac{\omega_\alpha(\alpha)}{\omega(\alpha)-\omega(\alpha')}-
\frac{1}{2}\cot{\frac{1}{2}(\alpha-\alpha')}\Big{]}d\alpha'\nonumber\\
&=&\frac{1}{2\pi i}\int_{\alpha-\pi}^{\alpha+\pi}
D_{\alpha'}\Big(\frac{ f(\alpha')}{\omega_\alpha(\alpha')}
\Big)D^{s-1}_\alpha\Big{[}\frac{\omega_\alpha(\alpha)}{\omega(\alpha)-\omega(\alpha')}
-\frac{1}{2}\cot{\frac{1}{2}(\alpha-\alpha')}\Big{]}d\alpha'\nonumber\\
&=& \frac{1}{2\pi
i}\int_{\alpha-\pi}^{\alpha+\pi}D_{\alpha'}\Big(\frac{
f(\alpha')}{\omega_\alpha(\alpha')}\Big)
D^{s-1}_\alpha\Big{[}\frac{\omega_\alpha(\alpha)}{\omega(\alpha)-\omega(\alpha')}-\frac{1}{\alpha-\alpha'}\Big{]}d\alpha'\nonumber\\
 &&-\frac{1}{2\pi i}\int_{\alpha-\pi}^{\alpha+\pi}
D_{\alpha'}\Big(
\frac{f(\alpha')}{2\omega_{\alpha}(\alpha')}\Big)D^{s-1}_\alpha
l\big(\frac{1}{2}(\alpha-\alpha')\big)d\alpha'.
\label{3.4}
\end{eqnarray}
Hence, we have
\begin{multline}
\label{3.5}
D^{s}_\alpha \big(\mathcal{K}[\omega^{(1)}] f
(\alpha)-\mathcal{K}[\omega^{(2)}] f
(\alpha)\big)\\
 =-\frac{1}{2\pi i}\int_{\alpha-\pi}^{\alpha+\pi}\Big(D_{\alpha'}\Big(\frac{ f(\alpha')}{\omega^{(1)}_\alpha(\alpha')}\Big)
D^{s-1}_\alpha\frac{q_2[\omega^{(1)}](\alpha,\alpha')}{q_1[\omega^{(1)}](\alpha,\alpha')}
-D_{\alpha'}\Big(\frac{
f(\alpha')}{\omega^{(2)}_\alpha(\alpha')}\Big)
D^{s-1}_\alpha\frac{q_2[\omega^{(2)}](\alpha,\alpha')}{q_1[\omega^{(2)}](\alpha,\alpha')}\Big)d\alpha'\\
 -\frac{1}{2\pi i}\int_{\alpha-\pi}^{\alpha+\pi}
\Big(D_{\alpha'}\Big(
\frac{f(\alpha')}{2\omega^{(1)}_{\alpha}(\alpha')}\Big)-D_{\alpha'}\Big(
\frac{f(\alpha')}{2\omega^{(2)}_{\alpha}(\alpha')}\Big)\Big)D^{s-1}_\alpha
l\big(\frac{1}{2}(\alpha-\alpha')\big)d\alpha'.
\end{multline}
Let us see the first part on the right side of (\ref{3.5}).
It can be
split as
\begin{multline}
\label{3.6}
\frac{1}{2\pi i}\int_{\alpha-\pi}^{\alpha+\pi}
D_{\alpha'}\Big(f(\alpha')
\frac{\omega^{(2)}_\alpha(\alpha')-\omega^{(1)}_\alpha(\alpha')}{\omega^{(1)}_\alpha(\alpha')\omega^{(2)}_\alpha(\alpha')}\Big)
D^{s-1}_\alpha \frac{q_2[\omega^{(1)}](\alpha,\alpha')}{
q_1[\omega^{(1)}](\alpha,\alpha')}d\alpha'\\
+\frac{1}{2\pi i}\int_{\alpha-\pi}^{\alpha+\pi}
D_{\alpha'}\Big(\frac{f(\alpha')}{\omega^{(2)}_\alpha(\alpha')}\Big)
D^{s-1}_\alpha\Big{[}
\frac{q_2[\omega^{(1)}-\omega^{(2)}] (\alpha,\alpha')}{
q_1[\omega^{(1)}](\alpha,\alpha')} \Big] d\alpha' \\
+\frac{1}{2\pi i}\int_{\alpha-\pi}^{\alpha+\pi}
D_{\alpha'}\Big(\frac{f(\alpha')}{\omega^{(2)}_\alpha(\alpha')}\Big)
D^{s-1}_\alpha\Big{[} \frac{q_2 [\omega^{(2)}] (\alpha, \alpha') }{
q_1 [\omega^{(2)}] (\alpha, \alpha') }
\frac{q_1[\omega^{(2)}-\omega^{(1)}] (\alpha,\alpha')}{
q_1[\omega^{(1)}](\alpha,\alpha')} \Big] d\alpha'.
\end{multline}
By Proposition \ref{prop2.6},  Lemma \ref{lem3.1},  Lemma \ref{lem3.5}
and Note \ref{note1.1},
the
$L^{\infty}$-norm of the first part of (\ref{3.6}) is bounded by
\begin{multline*}
C_1\Big\|\frac{f}{\omega^{(1)}_\alpha\omega^{(2)}_\alpha}(\omega^{(1)}_\alpha-\omega^{(2)}_\alpha)\Big\|_1\|\omega^{(1)}_\alpha\|_{s}
\exp\big({C_2}\|\omega^{(1)}_\alpha\|_{s-1}\big)\\
\leq
C_1\left(\|\tilde\theta^{(1)}\|_s+\|\tilde\theta^{(2)}\|_s+1\right)\exp\Big({C_2}\big(\|\tilde\theta^{(1)}\|_{s-1}+
\|\tilde\theta^{(2)}\|_{s-1}\big)\Big)\|\tilde\theta^{(1)}-\tilde\theta^{(2)}\|_1\|f\|_1,
\end{multline*}
with $C_1$  and $C_2$ depending on $s$.
For the second term in (\ref{3.6}), we use the Cauchy Schwartz inequality,
Lemmas \ref{lem3.1}, \ref{lem3.5} to obtain the bound as quoted in the
Lemma.
For the third term, we apply similar argument. We note that
for $ 0 \le l < s-1$,
\begin{multline*}
\left\| D_\alpha^{l} \left [ \frac{q_2 [\omega^{(2)}] (\alpha,
\alpha') }{ q_1 [\omega^{(2)}] (\alpha, \alpha') }  \right ]
D_\alpha^{s-1-l} \left [ \frac{q_1[\omega^{(2)}-\omega^{(1)}]
(\alpha,\alpha')}{
q_1[\omega^{(1)}](\alpha,\alpha')} \right ] \right\|_0  \\
\le \left|  D_\alpha^{l} \left [ \frac{q_2 [\omega^{(2)}] (\alpha,
\alpha') }{ q_1 [\omega^{(2)}] (\alpha, \alpha') }  \right ] \right
|_{\infty}\left\| D_\alpha^{s-1-l} \left [
\frac{q_1[\omega^{(2)}-\omega^{(1)}] (\alpha,\alpha')}{
q_1[\omega^{(1)}](\alpha,\alpha')} \right ] \right\|_0.
\end{multline*}
It is readily checked that
$$ D_{\alpha'}
\left [ \frac{q_2 [\omega^{(2)}] (\alpha, \alpha') }{ q_1
[\omega^{(2)}] (\alpha, \alpha') }  \right ] = - D_\alpha \left [
\frac{q_2 [\omega^{(2)}] (\alpha', \alpha) }{ q_1 [\omega^{(2)}]
(\alpha', \alpha) }  \right ]. $$ Since $ |. |_{\infty} \le C \| ,
\|_1$, it follows from Lemma \ref{lem3.5} that for $l < s-1$.
$$
\left| D_\alpha^{l} \left [ \frac{q_2 [\omega^{(2)}] (\alpha,
\alpha') }{ q_1 [\omega^{(2)}] (\alpha, \alpha') }  \right ]
\right|_{\infty} \le C_1\| \omega^{(2)}_\alpha \|_{s} \exp \left
(C_2 \| \omega^{(2)}_\alpha \|_{s-1} \right )
$$
with $C_1$  and $C_2$ depending on $s$. When $l=s-1$,
\begin{multline*}
\left \| D_\alpha^{s-1} \left [ \frac{q_2 [\omega^{(2)}] (\alpha,
\alpha') }{ q_1 [\omega^{(2)}] (\alpha, \alpha') }  \right ]
\frac{q_1[\omega^{(2)}-\omega^{(1)}] (\alpha,\alpha')}{
q_1[\omega^{(1)}](\alpha,\alpha')} \right \|_0  \\
\le \left| \left [ \frac{q_1[\omega^{(2)}-\omega^{(1)}]
(\alpha,\alpha')}{ q_1[\omega^{(1)}](\alpha,\alpha')} \right ]
\right|_{\infty} \left \| D_\alpha^{s-1} \left [ \frac{q_2
[\omega^{(2)}] (\alpha, \alpha') }{ q_1 [\omega^{(2)}] (\alpha,
\alpha') }  \right ]  \right\|_0.
\end{multline*}
Once again using Sobolev inequality $ |. |_\infty \le C \| .\|_1$ and
using Lemmas \ref{lem3.1}, \ref{lem3.5} we obtain
the stated bounds.

Since the function $l$ is symmetric about $\alpha$ and $\alpha'$, it
is easy to see that the stated bounds also hold for the second part
on the right side of (\ref{3.5}).

For $s\ge3$,
we use the more refined estimates in Lemma \ref{lem3.1} to obtain the third statement.

\end{proof}
\begin{lemma}{(See \cite{AD1} or appendix for proof)}
\label{lem3.9}
For  $\psi\in H^s\big(\mathbb{T}[0,2\pi]\big)$ with $s\geq 1$, the
operator $[\mathcal{H},\psi]$ is bounded from
$H^0\big(\mathbb{T}[0,2\pi]\big) $ to
$H^{s-1}\big(\mathbb{T}[0,2\pi]\big)$. And we have
\begin{eqnarray*}
\|[\mathcal{H},\psi]f\|_{s-1} \leq C \|f\|_0 \|\psi\|_s,
\end{eqnarray*}
 where $C$ depends on $s$.
\end{lemma}
\begin{lemma}
\label{lem3.10}
 For $s>\frac{1}{2}$
and  $\psi \in H^s\big(\mathbb{T}[0,2\pi]\big)$, the operator
 $[\mathcal{H},\psi]$ is bounded from $H^1\big(\mathbb{T}[0,2\pi]\big)$ to $H^s\big(\mathbb{T}[0,2\pi]\big)$, and
 \begin{eqnarray*}
 \|[\mathcal{H},\psi]f\|_s\leq C \|f\|_1\|\psi\|_s,
 \end{eqnarray*}
 where $C$ depends on $s$.
 \end{lemma}
 \begin{proof}
 We know that
 \begin{eqnarray}
 \|[\mathcal{H},\psi]f\|_s^2&=&\sum_{k\neq0}
 |k|^{2s}
 \big|\widehat{\mathcal{H}(\psi f)}(k)-\widehat{\psi \mathcal{H}f}(k)\big|^2+
 \big|\widehat{\mathcal{H}(\psi f)}(0)-\widehat{\psi \mathcal{H}f}(0)\big|^2.\nonumber
 \end{eqnarray}
 Since
 \begin{displaymath}
 \widehat{\mathcal{H}(\psi f)}(k)=(-i)\sgn(k)\widehat{\psi
 f}(k)=(-i)\sgn(k)\sum_{j=-\infty}^{\infty}\hat{\psi}(j)\hat{f}({k-j}),\mbox{
 for }k\neq 0,
 \end{displaymath}
 and
 \begin{displaymath}
 \widehat{\psi \mathcal{H}f}(k)=\sum_{j=-\infty}^{\infty}\hat{\psi}(j)\widehat{\mathcal{H}f}({k-j})=(-i)\sum_{j\neq
 k}\hat{\psi}(j)
 \sgn(k-j)\hat{f}({k-j}),
 \end{displaymath}
by Cauchy's inequality and the inequality $\|gh\|_0\leq
| h |_{\infty} \|g\|_0 \le C \| h\|_1 \|g\|_0$, we have
\begin{eqnarray*}
 &&\big\|[\mathcal{H},\psi]f\big\|_s^2\\
 &=&\sum_{k\neq 0}
 |k|^{2s}\Big|-i\sgn(k)\sum_{j=-\infty}^{\infty}\hat{\psi}(j)\hat{f}({k-j})+i\sum_{j\neq
 k}\hat{\psi}(j)
 \sgn(k-j)\hat{f}({k-j})\Big|^2\\
 &&+\Big|-i\sum_{j\neq 0}\hat{\psi}(j)
 \sgn(-j)\hat{f}({-j})\Big|^2\\
 &=&\sum_{k>0}|k|^{2s}\Big|2\sum_{j>k}
 \hat{\psi}(j)\hat{f}({k-j})+\hat{\psi}(k)
 \hat{f}(0)\Big|^2\\
 &&+\sum_{k<0}|k|^{2s}\Big|2\sum_{j<k}
 \hat{\psi}(j)\hat{f}({k-j})+\hat{\psi}(k)
 \hat{f}(0)\Big|^2
 +\Big|\sum_{j\neq 0}\hat{\psi}(j)
 \sgn(-j)\hat{f}(-j)\Big|^2 \\
&\leq&\sum_{k>0}8|k|^{2s}\Big|\sum_{j>k}\hat{
\psi}(j)\hat{f}({k-j})\Big|^2
 +\sum_{k<0}8|k|^{2s}\Big|\sum_{j<k}
\hat{\psi}(j)\hat{f}({k-j})\Big|^2 \\
&&+2\|\psi\|_s^2 |\hat{f}(0)|^2
 +\Big|\sum_{j\neq 0}\hat{\psi}(j)
 \sgn(-j)\hat{f}({-j})\Big|^2\\
 &\leq&\sum_{k>0}8\Big|\sum_{j>k}|j|^{s}\big|\hat{\psi}(j)\hat{f}({k-j})\big|\Big|^2
 +\sum_{k<0}8\Big|\sum_{j<k}|j|^s\big|\hat{\psi}(j)\hat{f}({k-j})\big|\Big|^2\\
 &&+2\|\psi\|_s^2 |\hat{f}(0)|^2+\|\psi\|_0^2\|f\|_0^2\\
&\leq&8\sum_{k=-\infty}^{\infty}\Big|\sum_{j=-\infty}^\infty\big||j|^s\hat{\psi}(j)\big|\big|\hat{f}({k-j})\big|\Big|^2
+3\|\psi\|_s^2\|f\|_0^2.
\end{eqnarray*}
We define
$$\big\{|j|^s|\hat{\psi}(s)|\big\}_{j\in\mathbb{Z}}=\Psi\mbox{
 and   }\big\{\hat{f}(j)\big\}_{j\in\mathbb{Z}}=\mbox{\boldmath
 $f$}.$$
 By Proposition 3.1999 in \cite{IO}, we know that $\|\mbox{\boldmath
 $f$}\ast\Psi\|_2\leq \|\mbox{\boldmath $f$}\|_1\|\Psi\|_2$. Hence
 we obtain the result of the lemma.

 \end{proof}

 \begin{lemma}
\label{lem3.11} If $f \in H^1\left( \mathbb{T} [0, 2\pi] \right) $,
$\omega^{(1)}$ and $\omega^{(2)}$ correspond to
$\tilde{\theta}^{(1)}$ and $\tilde{\theta}^{(2)}$ respectively, each
in $ \dot{H}^{1} $, $\|\tilde\theta^{(1)}\|_1$ and
$\|\tilde\theta^{(2)}\|_1<\epsilon_1$, then
\begin{eqnarray}
\|\mathcal{G}[\omega^{(1)}]f-\mathcal{G}[\omega^{(2)}]f\|_0 &\leq&
C_1
\|f\|_0\exp\left(C_2(\|\tilde\theta^{(1)}\|_1+\|\tilde\theta^{(2)}\|_1\right)
\|\tilde\theta^{(1)}-\tilde\theta^{(2)}\|_1.\nonumber
\end{eqnarray}
If  ${\tilde \theta}^1, {\tilde \theta}^2 \in \dot{H}^s$ for $s\geq
1$, then
\begin{eqnarray*}
&&\|\mathcal{G}[\omega^{(1)}]f-\mathcal{G}[\omega^{(2)}]f\|_s \\
&\leq& C_1\exp\Big({C_2}\big(\|\tilde\theta^{(1)}\|_{s}+
\|\tilde\theta^{(2)}\|_{s}\big)\Big)\|\tilde\theta^{(1)}-\tilde\theta^{(2)}\|_s\|f\|_1,\nonumber
\end{eqnarray*}
while    for $s\ge 3$,
\begin{eqnarray*}
&&\|\mathcal{G}[\omega^{(1)}]f-\mathcal{G}[\omega^{(2)}]f\|_s \\
&\leq& C_1\exp\Big({C_2}\big(\|\tilde\theta^{(1)}\|_{s-1}+
\|\tilde\theta^{(2)}\|_{s-1}\big)\Big)\Big(\big(\|\tilde\theta^{(1)}\|_s+\|\tilde\theta^{(2)}\|_s\big)\|\tilde\theta^{(1)}-\tilde\theta^{(2)}\|_{s-1}\\
&&\hspace{6cm}+
\|\tilde\theta^{(1)}-\tilde\theta^{(2)}\|_{s}\Big)\|f\|_1,
\end{eqnarray*} where the constants $C_1$ and $C_2$ depend on $s$
only.
\end{lemma}
\begin{proof}
From (\ref{2.5.0}), it follows that
\begin{multline*}
\| \mathcal{G}[\omega^{(1)}]f- \mathcal{G}[\omega^{(2)}]f  \|_0 \le
\left\| ( \omega^{(1)}_\alpha - \omega^{(2)}_\alpha ) \left [
\mathcal{H}, \frac{1}{\omega^1_\alpha} \right ] f \right\|_0 +
\left\| \omega^{2}_\alpha \left [ \mathcal{H},
\frac{1}{\omega^1_\alpha}-\frac{1}{\omega^2_\alpha} \right ]
f \right\|_0 \\
+ 2 \left\| ( \omega_\alpha^1 - \omega_\alpha^2 ) \mathcal{K}
[\omega^1] f \right\|_0 + + 2 \left\| \omega_\alpha^2 \left (
\mathcal{K} [\omega^1] - \mathcal{K} [\omega^2 ] \right ) f
\right\|_0.
\end{multline*}
Using Lemma \ref{lem3.1}, \ref{lem3.8}, \ref{lem3.10} and $\| h g
\|_0 \le | h|_\infty \| g \|_0 \le C \| h \|_1 \| g \|_0$, the first
statement holds.

 Now, consider
\begin{multline*}
\| \mathcal{G}[\omega^{(1)}]f- \mathcal{G}[\omega^{(2)}]f  \|_s
\le\left \| ( \omega^1_\alpha - \omega^2_\alpha ) \left [
\mathcal{H}, \frac{1}{\omega^1_\alpha} \right ] f \right\|_s +
\left\| \omega^{2}_\alpha \left [ \mathcal{H},
\frac{1}{\omega^1_\alpha}-\frac{1}{\omega^2_\alpha} \right ]
f \right\|_s \\
+ 2 \left\| ( \omega_\alpha^1 - \omega_\alpha^2 ) \mathcal{K}
[\omega^1] f \right\|_s + 2 \left\| \omega_\alpha^2 \left (
\mathcal{K} [\omega^1] - \mathcal{K} [\omega^2 ] \right ) f
\right\|_s.
\end{multline*}
Using Lemmas \ref{lem3.1}, \ref{lem3.8}, \ref{lem3.10}, using $\| h
g \|_s \le C_s \| h \|_s \| g \|_s$ for $s\ge 1$ and $\| h g \|_s
\le C_s (\| h \|_{s-1} \| g \|_s+\|h\|_s\|g\|_{s-1})$ for $s\ge 2$ ,
we see that the last two statements hold.

\end{proof}

\begin{prop}
\label{prop2.5} Assume $\tilde\theta\in \dot{H}^s$ for $s\ge3$. If
$\|\tilde \theta\|_1<\epsilon_1$, then for sufficiently small
$\epsilon_1$,
  there
exists unique solution $ \gamma \in \{u\in H^{s-2}
\big(\mathbb{T}[0,2\pi]\big)|\hat{u}(0)=0\}$ satisfying (\ref{2.6}).
This solution $\gamma$ satisfies the estimates
\begin{eqnarray}
\| \gamma \|_0  &\leq& \frac{C_0\sigma}{L} \| {\tilde \theta} \|_2,\nonumber \\
\|\gamma\|_{s-2}&\leq
&\frac{C_1\sigma}{L}\exp(C_2\|\tilde\theta\|_{s-2}) \|\tilde\theta\|_s, \nonumber\\
\left\| \gamma - \frac{2\pi}{L} \sigma \theta_{\alpha \alpha}
\right\|_{s} &\leq
&\frac{C_3\sigma}{L}\exp(C_4\|\tilde\theta\|_{s-1}) \| {\tilde
\theta} \|_{s} \| {\tilde \theta} \|_{3} ,\nonumber
\end{eqnarray}
where  $C_1$, $C_2$, $C_3$ and $C_4$ depend on $s$, but all are
independent of $L$. And for $s=3$, $C_2=0$.

If $\gamma^{(1)} $ and $\gamma^{(2)} $ correspond to $({\tilde
\theta}^{(1)}, L^{(1)} )$ and $({\tilde \theta}^{(2)} , L^{(2)} ) $,
each in $\mathcal{V}$, then for $3\le s\le r$,
$$ \| \gamma^{(1)} - \gamma^{(2)} \|_{s-2} \le C
\left ( \| {\tilde \theta}^{(1)} - {\tilde \theta}^{(2)} \|_s +
|L^{(1)}-L^{(2)}| \right ),
$$
$$ \left\| \gamma^{(1)} - \frac{2 \pi \sigma}{L^{(1)}} \theta^{(1)}_{\alpha \alpha}
- \gamma^{(2)} + \frac{2 \pi \sigma}{L^{(2)} }
\theta^{(2)}_{\alpha\alpha} \right\|_{s-2} \le C \left ( \| {\tilde
\theta}^{(1)} - {\tilde \theta}^{(2)} \|_{s-2} +  |L^{(1)} - L^{(2)}
| \right ),
$$
where $C$ depends on the diameter of $\mathcal{V}$ and $s$.
\end{prop}

\begin{proof}

From (\ref{Fomega0}),  since ${\hat \gamma} (0)=0$,
$\mathcal{F} [\omega_0] \gamma =0$. Therefore, (\ref{2.6}) implies
$$ \left [I + A_{\mu} \left ( \mathcal{F} [\omega] -\mathcal{F} [\omega_0]
\right ) \right ] \gamma = \frac{2\pi\sigma}{L} \theta_{\alpha \alpha}
$$
Therefore, if $\tilde\theta \in \dot{H}^2$, then Lemma \ref{lem3.11}
implies
$$ \| \mathcal{F} [\omega] \gamma -\mathcal{F} [\omega_0 ] \gamma \|_0
\le C \| {\tilde \theta} \|_1 \|\gamma \|_0.$$ where $C$ depends on
$\epsilon_1$. So, for sufficiently small $\epsilon_1$, if $\|\tilde
\theta\|_1 \leq \epsilon_1$, then
$$ \left [1 + A_{\mu} \left ( \mathcal{F} [\omega] -\mathcal{F}
[\omega_0]  \right ) \right ]^{-1} $$ exists  and from the bounds
above and Corollary \ref{cortheta},
$$ \| \gamma \|_0  \le \frac{C_0\sigma}{L} \| \theta_{\alpha \alpha} \|_0
\le \frac{C_0 \sigma}{L} \| {\tilde \theta} \|_2.$$ Further, we
obtain from the second part of Lemma \ref{lem3.11},
$$ \| \mathcal{F} [ \omega]\gamma  - \mathcal{F} [\omega_0] \gamma \|_{s-2}
\le C_1\exp(C_2\|\tilde\theta\|_{s-2}) \| {\tilde \theta} \|_{s-2}
\| \gamma \|_1,  $$ where $C_1$ and $C_2$ depend on  $s$. Therefore,
for $s \ge 3$, it follows from (\ref{2.6}) that
$$ \| \gamma \|_{s-2} \le \frac{2\pi \sigma }{L} \| {\tilde \theta } \|_{s}
+ C_1\exp(C_2\|\tilde\theta\|_{s-2}) \| {\tilde \theta} \|_{s-2} \|
\gamma \|_1 $$ which $C_1$ and $C_2$ depend on  $s$, which implies
for sufficiently small $\epsilon_1$ that the second statement holds.

 For the
third statement, we note that (\ref{2.6}) and the third part of
Lemma \ref{lem3.11} implies that
$$\left \| \gamma - \frac{2 \pi \sigma}{L} \theta_{\alpha \alpha}
\right\|_{s}\leq \frac{C_3\sigma}{L}
\exp(C_4\|\tilde\theta\|_{s-1})\| {\tilde \theta} \|_{s} \| {\tilde
\theta} \|_{3}, $$ where  $C_3$ and $C_4$ depend on  $s$.

From (\ref{2.6}), we obtain
$$ \| \gamma^{(1)} - \gamma^{(2)} \|_{s-2}  \le \left\|
\frac{2 \pi \sigma}{L^{(1)}} {\theta}^{(1)}_{\alpha \alpha} -
\frac{2 \pi \sigma}{L^{(2)}} {\theta}^{(2)}_{\alpha \alpha}
\right\|_{s-2} + \| \mathcal{F} [ \omega^{(1)} ] \gamma^{(1)} -
\mathcal{F} [ \omega^{(2)} ] \gamma^{(2)} \|_{s-2}. $$ Using
$$  \| \frac{1}{L^{(1)}} \theta^{(1)}_{\alpha \alpha} - \frac{1}{L^{(2)}}
\theta^{(2)}_{\alpha \alpha} \|_{s-2} \le
\frac{|L^{(1)}-L^{(2)}|}{L^{(1)} L^{(2)}} \| \theta^{(1)}_{\alpha
\alpha} \|_{s-2} + \frac{1}{L^{(2)}} \| \theta^{(1)}_{\alpha \alpha}
- \theta^{(2)}_{\alpha \alpha} \|_{s-2},
$$
and using Lemmas \ref{lem3.11} and the first part of the proposition,
\begin{multline}
\label{mathcalF} \| \mathcal{F} [\omega^{(1)} ] \gamma^{(1)} -
\mathcal{F} [\omega^{(2)} ] \gamma^{(2)} \|_{s-2} \le  \left \|
\mathcal{F} [\omega^{(1)} ] (\gamma^{(1)}-\gamma^{(2)} )
- \mathcal{F} [\omega_0 ] (\gamma^{(1)} - \gamma^{(2)} ) \right  \|_{s-2} \\
+ \left\| \mathcal{F} [\omega^{(1)} ] \gamma^{(2)} - \mathcal{F}
[\omega^{(2)} ]\gamma^{(2)}
\right\|_{s-2} \\
\le C \| {\tilde \theta}^{(2)} \|_{s-2} \| \gamma^{(1)} -
\gamma^{(2)} \|_1 + \frac{C}{L^{(2)}} \| {\tilde \theta}^{(1)} -
{\tilde \theta}^{(2)} \|_{s-2} \| {\tilde \theta}^{(2)} \|_3
\end{multline}
with $C$ depending on $s$ and the diameter of $\mathcal{V}$. The
fourth statement in the proposition follows since $({\tilde
\theta}^{(1)}, L^{(1)} ) , ({\tilde \theta}^{(2)}, L^{(2)}) \in
\mathcal{V} $. The fifth statement follows from (\ref{2.6}) by the
same set of arguments as above.

\end{proof}

\begin{lemma}
\label{lemUTtheta}  Assume $\tilde\theta\in\dot{H}^s$ for $s\ge3$.
If $\|\tilde\theta\|_1<\epsilon_1$, for sufficiently small
$\epsilon_1$, the corresponding $U$ and $T$ in (\ref{2.2}) and
(B.3), with ${\hat \theta} (1)$ and ${\hat \theta} (-1)$ determined
from ${\tilde \theta}$ using (B.4), satisfies the following
estimates:
$$ \left\| U - \frac{2 \pi^2 \sigma}{L^2} \mathcal{H} [\theta_{\alpha\alpha} ]
\right\|_0 \le \frac{C_1 \sigma }{L^2} \| {\tilde \theta} \|_1 \|
{\tilde \theta} \|_2,$$
$$ \left\| U - \frac{2 \pi^2 \sigma}{L^2} \mathcal{H} [\theta_{\alpha \alpha}]
\right\|_{s-2} \le \frac{C_2\sigma}{L^2}
\exp(C_3\|\tilde\theta\|_{s-2})\| {\tilde \theta}
\|_{s-2}\|\tilde\theta\|_{3},$$
$$ \left\| U - \frac{2 \pi^2 \sigma}{L^2} \mathcal{H} [\theta_{\alpha \alpha}]
\right\|_{s} \le \frac{C_2\sigma}{L^2}
\exp(C_3\|\tilde\theta\|_{s-1})\| {\tilde \theta}
\|_{s}\|\tilde\theta\|_{3},$$
$$\|U\|_{s-2}\le \frac{C_2\sigma}{L^2}
\exp(C_3\|\tilde\theta\|_{s-2})\|\tilde\theta\|_s,$$
$$ \| T \|_{s-1} \le  \frac{C_2 \sigma}{L^2}\exp(C_3\|\tilde\theta\|_{s})  \| {\tilde \theta} \|_{s}, $$
where $C_1$ depends on $\epsilon_1$, $C_2$ and $C_3$ depend on $s$.

 If $U^{(1)}$ and $U^{(2)}$ (or $T^{(1)}$, $T^{(2)}$) correspond respectively
to $\left ( {\tilde \theta}^{(1)}, L^{(1)} \right ) $ and $\left (
{\tilde \theta}^{(2)}, L^{(2)} \right ) $, each in $\mathcal{V} $,
then for  $r\ge3$,
$$ \| U^{(1)} - U^{(2)} \|_{r-2}  \le  C_4 \left (
\| {\tilde \theta}^{(1)} - {\tilde \theta}^{(2)}  \|_{r} + |L^{(1)}
- L^{(2)} |  \right ), $$
\begin{multline*}
\left\| U^{(1)} - \frac{2 \pi^2}{(L^{(1)})^2} \mathcal{H}
[\theta^{(1)}_{\alpha \alpha} ] - U^{(2)} - \frac{2
\pi^2}{(L^{(2)})^2} \mathcal{H} [\theta^{(2)}_{\alpha \alpha} ]
\right\|_{r-2} \\
\le C_4\left (  \| {\tilde \theta}^{(1)}\|_r+ \| {\tilde
\theta}^{(2)}\|_r\right)\left(\| {\tilde \theta}^{(1)} - {\tilde
\theta}^{(2)}  \|_{r-2} + |L^{(1)} - L^{(2)} | \right ),
\end{multline*}
$$ \| T^{(1)} - T^{(2)} \|_{r-1}  \le C_4\left ( \|
{\tilde \theta}^{(1)} - {\tilde \theta}^{(2)}  \|_{r} + |L^{(1)} -
L^{(2)} | \right ),$$ where $C_4$ depends on the diameter of
$\mathcal{V}$ and $r$.
\end{lemma}
\begin{proof}
From (\ref{2.5.0}) and (\ref{omega0f}), Lemma \ref{lem3.11}
it follows that
$$ \left\| U - \frac{\pi}{L} \mathcal{H} \gamma \right\|_0 \le
\frac{\pi}{L} \| \mathcal{G} [\omega] \gamma - \mathcal{G}
[\omega_0] \gamma \|_0 \le  \frac{C_1}{L}\| {\tilde \theta} \|_1
\|\gamma\| _0, $$with $C_1$ depending on $\epsilon_1$. Using
Proposition \ref{prop2.5}, we obtain
$$ \left\| U - \frac{2 \pi^2 \sigma}{L^2} \mathcal{H} [\theta_{\alpha\alpha} ]
\right\|_0 \le \frac{C_1\sigma}{L^2} \| {\tilde \theta} \|_1
\|{\tilde \theta} \|_2
$$
with $C_1$ depending on $\epsilon_1$. Again from (\ref{2.5.0}) and
(\ref{omega0f}), Lemma \ref{lem3.11} and Proposition \ref{prop2.5},
we obtain
\begin{eqnarray*}
\left\| U -\frac{2 \pi^2 \sigma}{L^2} \mathcal{H} [ \theta_{\alpha
\alpha} ] \right\|_{s} &\le& \frac{\pi}{L}
\| \mathcal{G} [\omega] \gamma - \mathcal{G} [\omega_0] \gamma \|_{s} +\frac{\pi}{L}\left\|\mathcal{H}
\left[\gamma-\frac{2\pi\sigma}{L}\theta_{\alpha\alpha}\right]\right\|_s\\
&\le& \frac{C_2}L\exp(C_3\|\tilde\theta\|_{s-1})\|  {\tilde \theta}
\|_{s} \|\gamma\| _1 \le \frac{C_2
\sigma}{L^2}\exp(C_3\|\tilde\theta\|_{s-1})  \|{\tilde \theta}
\|_{s} \| {\tilde \theta} \|_3,
\end{eqnarray*}
where $C_2$ and $C_3$ depend on $s$. Similarly, we can get the
second and fourth statements. This gives all the desired results for
$U$ in terms of ${\tilde \theta}$.

Again, from noting that the second equation from (B.3), and the
above estimates on $U$, we obtain
\begin{eqnarray}\label{T0}
\|T\|_0\leq C\|U(1+\theta_\alpha)\|_1\leq
\frac{C_2\sigma}{L^2}\exp(C_3\|\tilde\theta\|_3)\|\tilde\theta\|_3,
\end{eqnarray}
and
\begin{eqnarray*}
&&\left \| T_\alpha
- \frac{2 \pi^2 \sigma}{L^2} \mathcal{H} [\theta_{\alpha\alpha} ] \right\|_{s-2} \\
&\le& \left\| U - \frac{2 \pi^2 \sigma}{L^2} \mathcal{H}
[\theta_{\alpha\alpha} ] \right\|_{s-2}
+ \| U \theta_\alpha \|_{s-2}  \\
&\le& \frac{C_2 \sigma}{L^2}\exp(C_3\|\tilde\theta\|_{s}) \left [
\|{\tilde \theta} \|_{s-2} \|{\tilde \theta} \|_3 + \| {\tilde
\theta}\|_{s}\| {\tilde \theta}\|_{s-1} \right ],
\end{eqnarray*}
where $C_2$ and $C_3$ depend on $s$. Hence the fourth statement
holds.

 Also, we obtain from
(\ref{2.5.0}), (\ref{omega0f}),
\begin{multline*}
\| U^{(1)} - \frac{2 \pi^2}{(L^{(1)})^2} \mathcal{H}
[\theta^{(1)}_{\alpha\alpha} ] - U^{(2)} + \frac{2
\pi^2}{(L^{(2)})^2} \mathcal{H} [\theta^{(2)}_{\alpha\alpha} ]
\|_{r-2} \le \| \frac{\pi}{L_1} \mathcal{G} [\omega^{(1)} ]
\gamma^{(1)}
- \frac{\pi}{L_2} \mathcal{G} [\omega^{(2)} ] \gamma^{(2)} \|_{r-2} \\
\le  \frac{|L^{(1)}-L^{(2)}|}{L^{(1)} L^{(2)}} \| \mathcal{G}
[\omega^{(1)}] \gamma^{(1)} - \mathcal{G} [\omega_0] \mathcal
\gamma^{(1)} \|_{r-2} + \frac{C}{L^{(2)}} \| \mathcal{G} [\omega^{(1)}
] \gamma^{(1)} - \mathcal{G} [\omega^{(2)} ] \gamma^{(1)} \|_{r-2}
\\
+ \frac{C}{L^{(2)}} \left\| \mathcal{G} [\omega^{(2)} ] \left
(\gamma^{(1)}-\gamma^{(2)} \right ) \right\|_{r-2}.
\end{multline*}
The stated results on differences between $U^{(1)} $, $U^{(2)}$
follow from Lemma \ref{lem3.11} and Proposition \ref{prop2.5} on
using the condition that each of $\left ( {\tilde \theta}^{(1)} ,
L^{(1)} \right ), \left ( {\tilde \theta}^{(2)} , L^{(2)} \right )
\in \mathcal{V} $. We note the second equation from (B.3), so the
stated results follow for $T^{(1)} - T^{(2)}$  as well.
\end{proof}

\section{Energy estimate}

 We define  energy  we will use is the $H^r\big(\mathbb{T}[0,2\pi]\big)$ norm
of $\Tilde\theta_n$; it is defined by
\begin{displaymath}
E_n(t)=\frac{1}{2}\int_0^{2\pi}(D^r\Tilde\theta_n)^2d\alpha.
\end{displaymath}

We first need to estimate the following terms in the evolution
equations.

\begin{lemma}
\label{prel2.14} Assume $ X_n = \left ( {\tilde \theta}_n, L_n
\right )$ is a solution to the initial value problem (\ref{2.12})
with ${\tilde \theta}_n \in \mathcal{B}$ for $r \ge 3$. If the size
$\epsilon$ of the ball $\mathcal{B}$ is small enough, then ${\tilde
\theta}_n (., t) \in \mathcal{B} $ for all $t$ for which solution
exists. Further, the corresponding energy $E_n$, as defined above,
satisfies the inequality
$$ \frac{d E_n}{dt} \le -\frac{\pi^2 \sigma}{L_n^2} E_n. $$
\end{lemma}

\begin{proof}
For $r\geq 3$, taking the derivative of $E_n(t)$ with respect to
$t$, we have
\begin{displaymath}
\frac{d}{dt}E_n(t)=\int_0^{2\pi}(D^r\Tilde\theta_n)
(D^r\Tilde\theta_{n,t})d\alpha.
\end{displaymath}
Using (C.1), (C.2) and (\ref{2.11}), on integration by parts we find
\begin{eqnarray*}
\frac{d}{dt} E_n &=& I_1 + I_2 + I_3 + I_4, {\rm where} \\
I_1 &=& -\int_0^{2\pi} D^{r+1} {\tilde \theta}_n
D^r \left ( P_n U_n \right ) d\alpha, \\
I_2 &=& \int_0^{2\pi} D^r {\tilde \theta}_n D^{r-1}
\left ( P_n U_n \right )  d\alpha, \\
I_3 &=& \int_0^{2\pi} D^r {\tilde \theta}_n D^{r-1}
P_n \left ( \theta_{n,\alpha} U_n \right )  d\alpha, \\
I_4 &=& \int_0^{2\pi} D^r {\tilde \theta}_n D^{r} P_n \left ( T_n
\theta_{n,\alpha} \right )  d\alpha.
\end{eqnarray*}
On using ${\tilde \theta}_n = P_n \theta_n $, we can rewrite
\begin{equation*}
I_1 = -\frac{2 \pi^2 \sigma}{L^2_n} \int_0^{2 \pi} D^{r+1} {\tilde
\theta}_n D^{r+2} \mathcal{H} [ {\tilde \theta}_{n} ] d\alpha -
\int_0^{2\pi} D^{r+1} {\tilde \theta}_n P_n D^r \left [ U_n -
\frac{2 \pi^2 \sigma}{L_n^2} \mathcal{H} [ \theta_{n,\alpha \alpha}
]  \right ] d \alpha.
\end{equation*}
Using Lemma \ref{lemUTtheta} to bound the second term $I_1$, it
follows from Cauchy-Scwartz inequality that
$$ I_1 \le - \frac{2\pi^2 \sigma}{L_n^2} \| {\tilde \theta}_n
\|_{r+3/2}^2 +
\frac{C_1\sigma}{L_n^2}\exp(C_2\|\tilde\theta_n\|_{r-1} )\| {\tilde
\theta}_{n} \|_{r} \| {\tilde \theta} \|_{r+1}\|\tilde\theta_n\|_3,
$$
where $C_1$ and $C_2$ depend on $s$. Consider $I_2$. Applying Lemma
\ref{lemUTtheta} once again, we obtain
\begin{eqnarray*}
I_2 &=&  \frac{2 \pi^2 \sigma}{L_n^2} \int_0^{2\pi} D^r {\tilde
\theta}_n D^{r+1} \mathcal{H} [{\tilde \theta}_n ] d\alpha+
\int_0^{2\pi} D^r {\tilde \theta}_n D^{r-1} P_n \left [ U_n -
\frac{2 \pi \sigma}{L_n^2} \mathcal{H} [\theta_{n,\alpha\alpha}]
\right ] d\alpha \\
&\le & \frac{2 \pi^2 \sigma}{L_n^2} \| {\tilde \theta}_n \|^2_{r+1/2}
+ \frac{C_1\sigma}{L_n^2}\exp(C_2\|\tilde\theta_n\|_{r-1} ) \|
{\tilde \theta}_n \|_r^2 \| {\tilde \theta}_n \|_3,\\
I_3&=& \frac{2 \pi^2 \sigma}{L_n^2} \int_0^{2\pi} D^r {\tilde
\theta}_n D^{r-1}P_n\big(\theta_{n,\alpha} \mathcal{H}
[{\theta}_{n,\alpha\alpha} ] \big)d\alpha\\
&&+ \int_0^{2\pi} D^r {\tilde \theta}_n D^{r-1} P_n \left
\{\theta_{n,\alpha}\left[ U_n - \frac{2 \pi \sigma}{L_n^2}
\mathcal{H} [\theta_{n,\alpha\alpha}]\right]
\right \} d\alpha \\
&\le & \frac{C_1 \sigma}{L_n^2} \| {\tilde \theta}_n
\|_{r+1}\|\tilde\theta_n\|_{r}\|\tilde\theta_n\|_{r-1} +
\frac{C_1\sigma}{L_n^2}\exp(C_2\|\tilde\theta_n\|_{r-1} ) \| {\tilde
\theta}_n \|_r^2 \| {\tilde \theta}_n
\|_3,\\
\end{eqnarray*}
\begin{eqnarray*}
I_4&=& \int_0^{2\pi} D^r {\tilde \theta}_n  P_n \big(\sum_{j=0}^r
C_{r,j} D^{j} T_n D^{r+1-j} \theta_{n}
 \big) d\alpha\\
&=& \int_0^{2\pi} D^r {\tilde \theta}_n  P_n  (T_n D^{r+1}
\theta_{n}  ) d\alpha+ C_{r,1}\int_0^{2\pi} D^r {\tilde \theta}_n
P_n \left (  T_{n,\alpha} D^{r} \theta_{n}
\right ) d\alpha\\
&&+ \int_0^{2\pi} D^r {\tilde \theta}_n  P_n \left (\sum_{j=2}^r
C_{r,j} D^{j-2} \left(U_n(1+\theta_{n,\alpha})\right) D^{r+1-j}
\theta_{n}
\right ) d\alpha\\
&\leq&
\frac{C_1\sigma}{L_n^2}\left(\exp(C_3\|\tilde\theta_n\|_3)\|\tilde\theta_n\|_3\|\tilde\theta_n\|_{r+1}^2
+\exp(C_2\|\tilde\theta_n\|_{r-1} ) \| {\tilde \theta}_n \|_r^2\|
{\tilde \theta}_n \|_{r-1}\right),
\end{eqnarray*} where $C_1$ and $C_2$ depend on $s$.
Adding up $I_1$ through $I_4$, using $({\tilde \theta}_n, L_n) \in
\mathcal{V}$ and the fact that $\| {\tilde \theta}_n \|^2_{r+1/2}
\le \frac{1}{4} \| {\tilde \theta}_n \|^2_{r+3/2} $ since the
Fourier 0 and $\pm1$ modes for $\tilde\theta_n$ are zero, we obtain
for $r=3$,
\begin{multline}
\label{energybounds} \frac{d}{dt} E_n  \le - \frac{3 \pi^2 \sigma}{2
L_n^2} \| {\tilde \theta}_n \|^2_{r+3/2} + \frac{C\sigma}{L_n^2} \|
{\tilde \theta}_n \|_{r+1}^2 \| {\tilde \theta}_n \|_{3}\\ \le -
\frac{\sigma}{L_n^2} \| {\tilde \theta}_n \|^2_{r+3/2} \left (
\frac{3}{2} \pi^2 - C \| {\tilde \theta}_n \|_{3} \right )
 \le - \frac{3\pi^2 \sigma}{2L_n^2} E_n\left (
1 - \frac{2C}{3\pi^2} (2E_n)^{1/2} \right ),
\end{multline}
and for $r>3$,
\begin{multline}
\label{energybounds2} \frac{d}{dt} E_n  \le - \frac{3 \pi^2
\sigma}{2 L_n^2} \| {\tilde \theta}_n \|^2_{r+3/2} +
\frac{C\sigma}{L_n^2} \| {\tilde \theta}_n \|_{r+1}^2 \| {\tilde
\theta}_n \|_{r-1}\\ \le - \frac{\sigma}{L_n^2} \| {\tilde \theta}_n
\|^2_{r+3/2} \left ( \frac{3}{2} \pi^2 - C \| {\tilde \theta}_n
\|_{r-1} \right )
 \le - \frac{3\pi^2 \sigma}{2L_n^2} E_n\left (
1 - \frac{2C}{3\pi^2} (2E_n)^{1/2} \right ),
\end{multline}
 where   $C=C_1\exp(C_2\|\tilde\theta\|_{r-1})$ with
$C_1$ and $C_2$ depending on $r$. It immediately follows that if $1
- C\frac{2}{3\pi^2} (2E_n)^{1/2}
> 0$ initially, then $E_n (t)$ decreases in time and $E_n (t) \le E_n (0)$ for
all $t$. This implies that for small enough $\epsilon$, if ${\tilde
\theta}_n \in \mathcal{B} $ initially, it remains there for any $t$
for which the solution exists. More, generally, we have
$$ \frac{d E_n}{dt} \le -\frac{ \pi^2 \sigma}{L_n^2} E_n. $$
\end{proof}
\begin{corollary}
\label{coro4.2} Assume $  \left ( {\tilde \theta}_n, L_n \right )$
is a solution to the initial value problem (\ref{2.12}) with
${\tilde \theta}_n \in \mathcal{B}$ with $r \ge 3$. Then for
sufficiently small ball size $\epsilon$ of $\mathcal{B}$, we have
$$ \frac{d E_n}{dt} \le -\frac{\pi^2 \sigma}{L_n^2} \|\tilde\theta_n\|_{r+3/2}^2,$$
$$\frac{d \|\tilde\theta_n\|_{r+1}^2}{dt} \le -\frac{\pi^2 \sigma}{L_n^2} \|\tilde\theta_n\|_{r+1}^2.$$
\end{corollary}
\begin{proof} The proof of the first statement comes from
(\ref{energybounds}) and (\ref{energybounds2}).

Replacing $r$ by $r+1$ in (\ref{energybounds2}), we obtain
\begin{multline}
\label{corof4.2} \frac{d \|\tilde\theta_n\|_{r+1}^2}{dt}\le -
\frac{3 \pi^2 \sigma}{2 L_n^2} \| {\tilde \theta}_n \|^2_{r+5/2} +
\frac{C\sigma}{L_n^2} \| {\tilde \theta}_n \|_{r+2}^2 \| {\tilde
\theta}_n \|_r\\ \le - \frac{3\sigma\pi^2}{2L_n^2} \| {\tilde
\theta}_n \|^2_{r+5/2} \left ( 1 - \frac{2C}{3\pi^2} \| {\tilde
\theta}_n \|_r \right ),
\end{multline}
where $C=C_1\exp\left(C_2\|\tilde\theta_n\|_r\right)$ with $C_1$ and
$C_2$ depending only on $r$. Hence for small enough $\epsilon$, if
$\tilde\theta_n\in\mathcal{B}$, then by (\ref{coro4.2}), we have
$$\frac{d \|\tilde\theta_n\|_{r+1}^2}{dt}\le  - \frac{\sigma\pi^2}{L_n^2} \| {\tilde
\theta}_n \|^2_{r+1}.$$
\end{proof}

\begin{prop}
\label{prop4.2} Let $({\tilde \theta}_n, L_n ) $ be continuous
solution to (C.1) and (C.2) with ${\tilde \theta}_n (t) \in
\mathcal{B}$, with $r\geq 3$ and with initial conditions
(\ref{2.12}) Then for sufficiently small ball size $\epsilon$ of
$\mathcal{B}$, as long as solution exists,
\begin{eqnarray}
\label{bounds_E}
E_n (t) \le E_n (0) \exp \left [-\frac{\sigma t}{18}  \right ],
\end{eqnarray}
\begin{eqnarray}
\label{bounds_r+1}
\|\tilde\theta_n(\cdot,t)\|^2_{r+1}\leq\|\tilde\theta_n(\cdot,0)\|^2_{r+1}
\exp \left [-\frac{\sigma t}{18}  \right ], \label{r+1bounds}
\end{eqnarray}
\begin{eqnarray}
\label{4.4}
\big|L_n^3 (t)- 8 \pi^3 \big|& \leq&
C\sqrt{E_n(0)}\Big(1-\exp(-\frac{1}{18}\sigma t)\Big) ,
\label{4.5}
\end{eqnarray}
 where $C$ depends on the diameter of $\mathcal{B}$, not on $n$.
\end{prop}
\begin{proof}
We note from the evolution equation for $L_n$ may be rewritten as
$$ L_n^2\frac{d L_n}{dt} = -{L_n^2}\int_0^{2 \pi} \left [ U_n
- \frac{2\pi^2}{L_n^2} \mathcal{H} [ {\theta}_{n,\alpha\alpha} ]
\right ] d\alpha - {L_n^2}\int_0^{2\pi} U_n
\theta_{n,\alpha} d\alpha.$$ Using Lemma \ref{lemUTtheta},
 on
integration, it follows that
\begin{equation}
\label{bounds_L} |L_n^3 (t) - (2 \pi)^3 | \le C
\int_0^t\| {\tilde \theta}(\cdot,t') \|_3^2dt'\leq C \int_0^t E_n
(t') dt',
\end{equation}where $C$ depends on the diameter of $\mathcal{B}$.
Since $E_n (t) \le E_n (0)$, it follows that
$$ |L_n^3 (t) | \le 8 \pi^3 + C E_n (0) t ,$$
where $C$ depends on the diameter of $\mathcal{B}$. Using Lemma
\ref{prel2.14}, we obtain preliminary estimates:
$$ E_n (t) \le E_n (0) \exp \left \{- \frac{3 \sigma \pi^2}{C E_n (0)}
\left[\left (8 \pi^3 + C E_n (0) t \right )^{1/3} -2\pi \right
]\right\}.
$$ Going back to (\ref{bounds_L}),
it follows that for sufficiently small $E_n (0)$, for any $t$,
\begin{eqnarray}\label{boundL}
|L_n^3 (t) - 8 \pi^3 | < 1 \end{eqnarray} which implies that $L_n$
cannot escape the interval $ ( 2\pi -1, 2 \pi + 1) $. Going back to
Lemma \ref{prel2.14}, this implies that
$$ \frac{d}{dt} E_n \le -\frac{\pi^2 \sigma}{(2 \pi + 1)^2} E_n
\le -\frac{\sigma}{18} E_n$$ and therefore (\ref{bounds_E}) follows.
(\ref{bounds_r+1}) follows from Corollary \ref{coro4.2} once we use
(\ref{boundL}). Furthermore, plugging   estimates (\ref{bounds_E})
into (\ref{bounds_L}), we have
$$ |L^3_n (t) - 8 \pi^3 | \le \frac{18 C E_n (0)}{\sigma}
\left [ 1 - \exp \left ( -\frac{\sigma t}{18} \right ) \right ]. $$
\end{proof}

\noindent{\bf Proof of Proposition \ref{prop2.14}}: This follows
readily from Lemma \ref{prel2.14} and Proposition \ref{prop4.2},
since Lemma \ref{prel2.14} assures that as long as solution $X_n $
to (\ref{2.12}) exists, corresponding ${\tilde \theta}_n$ does not
exit the ball $\mathcal{B}$ and therefore Proposition \ref{prop4.2}
can be applied to obtain estimates on $E_n (t)$ and $L_n (t)$.

\section{Existence of Solutions}
In this section, we demonstrate existence of solutions to initial
value problem (\ref{2.12}). We then show that these solutions
converge (as the truncation $n$ tends to $\infty$ ) to a solution of
(B.1), (B.3) and (B.4) with initial condition (\ref{2.4}). We
demonstrate that this solution to (B.1), (B.3) and (B.4) with
initial condition (\ref{2.4}) is unique and has the same regularity
as the initial data.

\begin{definition}
\label{def5.1}
We define
$$\||X\||=\|u\|_r+|v|$$ for $X=(u,v)\in
H^r\big(\mathbb{T}[0,2\pi]\big)\times\mathbb{R}$.
\end{definition}
\noindent{\bf Proof of Proposition \ref{prop2.13}:} First we show
that the operator $F_n : \mathcal{V}\rightarrow
H^r\big(\mathbb{T}[0,2\pi]\big)\times\mathbb{R}$ is bounded, i.e.
$\|F_{n,1}\|_r+|F_{n,2}| < \infty, \forall X_n\in \mathcal{V}$. It
follows from Lemma \ref{lemUTtheta} that
\begin{eqnarray*}
\big\|P_n U_{n,\alpha}+P_n T_n(1+\theta_{n,\alpha})\big\|_{r}
&\leq&  \| U_{n,\alpha} \|_{r}+ \|T_n \|_r + \|T_n \|_r \|\theta_{n,\alpha}
\|_r  \\
&\leq& C \left (
\|{\tilde \theta}_n \|_{r+3} + \| {\tilde \theta}_{n}\|_{r+1}
+ \| {\tilde \theta}_n \|_{r+1}^2 \right ) \le C n^3
\| {\tilde \theta}_n \|_r  , \\
| F_{n,2} | \le \|1+ \theta_{n,\alpha}\|_0\| U_n\|_0 &\leq& C \|
{\tilde \theta}_n \|_2 \left ( 1 + \| {\tilde \theta}_n \|_1 \right
),
\end{eqnarray*}
where $C$ depends on $n$, $r$ and the diameter of $\mathcal{V}$.

Consider $X_n^{(1)}, X^{(2)}_n \in \mathcal{V}$. We have
\begin{multline}
\label{5.1} \|F_{n,1}(X_n^{(1)})-F_{n,1}(X_n^{(2)})\|_r\leq
\Big\|\big(\frac{2\pi}{L_n^{(1)}}-\frac{2\pi}{L^{(2)}_n}\big)P_n\big(U_{n,\alpha}^{(1)}+T_n^{(1)}(1+\theta_{n,\alpha}^{(1)})\big)\Big\|_r\\
+\frac{2\pi}{L^{(2)}_n}\big\|P_n\big(U_{n,\alpha}^{(1)}-U^{(2)}_{n,\alpha}\big)\big\|_r+
\frac{2\pi}{L^{(2)}_n}\Big\|P_n\big(T_n^{(1)}(1+\theta_{n,\alpha}^{(1)})
-T^{(2)}_n(1+\theta^{(2)}_{n,\alpha})\big)\Big\|_r.\end{multline} It
follows from Lemma \ref{lemUTtheta} that
\begin{eqnarray}
\Big\|\big(\frac{2\pi}{L_n^{(1)}}-\frac{2\pi}{L^{(2)}_n}\big)P_n\big(U^{(1)}_{n,\alpha}+T^{(1)}_n(1+\theta^{(1)}_{n,\alpha})\big)\Big\|_r \\
\leq C n^3  [ |L_n^{(1)} - L_n^{(2)}|  \le c
\||X^{(1)}_n-X^{(2)}_n\||,\nonumber
\end{eqnarray}
where $c$ depends on $n$, $r$ and the diameter of $\mathcal{V}$.
Further, using Lemma \ref{lemUTtheta}
\begin{eqnarray*}
| F_{n,2}^{(2)} - F_{n,2}^{(2)} | &\le& C \left(\| U_n^{(1)} -
U_n^{(2)} \|_0  \big ( 1 + \| {\tilde \theta}_n^{(1)} \|_1 \big)+ \|
U_n^{(2)} \|_0
\|{\tilde \theta}^{(1)} - {\tilde \theta}^{(2)} \|_1\right) \\
&\le& C \left (  |L_1-L_2| + \| {\tilde \theta}^{(1)}-{\tilde
\theta}^{(2)} \|_1 \right ) \le c \|| X^{(1)}_n - X^{(2)}_n \||,
\end{eqnarray*}
where $c$ depends on $n$, $r$ and the diameter of $\mathcal{V}$.
Therefore, from ODE theory, it follows that there exists local
solution $X_n \in C^1 \big ( [0, S_n]; \mathcal{V} \big )$ over some
time interval $S_n$ that may depend on $n$, $r$ and $\epsilon$.

\begin{lemma}
\label{lem5.2} There exists sufficiently small $\epsilon>0$  such
that solutions $X_n=(\Tilde\theta_n,L_n)\in
C^1\left([0,S];\mathcal{V}\right)$ of the initial value problem
(\ref{2.12}) form a Cauchy sequence in
$C\left([0,S];\dot{H}^1\times\mathbb{R}\right)$ for any $S>0$.
\end{lemma}
\begin{proof}
We define difference energy function $E_{mn}$ as
$$E_{mn}=E_{mn}^1+(L_n-L_m)^2$$
where
$E_{mn}^1=\frac{1}{2}\int_0^{2\pi}
\big(D (\Tilde\theta_n-\Tilde{\theta}_m)\big)^2 d\alpha.$
Notice that $E_{mn} (0)=E^1_{mn} (0)$.
Without loss of generality, we assume $m > n$ as otherwise we can
switch the role of $m$ and $n$ in the ensuing argument.

Using the first equation in (C.1),
\begin{multline}
\label{5.2}
\frac{dE_{mn}^1}{dt}
=\int_0^{2\pi}D (\Tilde\theta_n-\Tilde{\theta}_m)D^2
\big(\frac{2\pi}{L_n}P_nU_n-\frac{2\pi}{L_m}P_mU_m\big)d\alpha\\
+\int_0^{2\pi}D (\Tilde\theta_n-\Tilde{\theta}_m) D
\Big(\frac{2\pi}{L_n}P_n\big(T_n(1+\theta_{n,\alpha})
\big)-\frac{2\pi}{L_m}P_m\big(T_m(1+\theta_{m,\alpha})\big)\Big)d\alpha
\equiv I_1 + I_2.
\end{multline}
Defining ${\tilde \theta}_{nm} = {\tilde \theta}_n - {\tilde \theta}_m$,
it is clear that
\begin{eqnarray*}
I_1 &=& -2 \pi \left ( \frac{1}{L_n} - \frac{1}{L_m}\right )
\int_0^{2\pi} D^2 {\tilde \theta}_{nm} P_n D U_n d\alpha +
\frac{2\pi}{L_m} \int_0^{2\pi} D {\tilde \theta}_{mn} (P_n - P_m ) D^2 U_n \\
&+& \frac{2\pi}{L_m} \int_0^{2\pi} D {\tilde \theta}_{mn} P_m D^2
(U_n-U_m ) \equiv I_{1,1}+I_{1,2}+I_{1,3}
\end{eqnarray*}
From estimates in Lemma \ref{lemUTtheta} and restrictions due to
$\left ( {\tilde \theta}_n , L_n \right ), \left ( {\tilde
\theta}_m, L_n \right ) \in \mathcal{V}$, we obtain
$$ |I_{1,1} | \le c \epsilon E_{mn}^{1/2} \| {\tilde \theta}_{nm} \|_2, $$
where $c$ depends on the diameter of $\mathcal{V}$.
 We note that since $P_n \theta_n = {\tilde \theta}_n $ and $P_m
\theta_n = {\tilde \theta}_n $, as $m > n$, we can write $I_{1,2}$
$$ I_{1,2} = \frac{2\pi}{L_m} \int_0^{2\pi}
D {\tilde \theta}_{mn}  D^2 [P_n - P_m ] \left (U_n - \frac{2 \pi^2
\sigma}{L_n^2} \mathcal{H} [{ \theta}_{n,\alpha\alpha} ] \right )
d\alpha.$$ Therefore, using Lemma \ref{lemUTtheta},
$$ |I_{1,2} | \le \frac{c}{n} E_{mn}^{1/2}\left \|  U_n - \frac{2 \pi^2 \sigma}{L_n^2} \mathcal{H} [
\theta_{n,\alpha\alpha} ]  \right\|_3\le \frac{C \epsilon}{n}
E_{mn}^{1/2},
$$
where $C$ depends on the diameter of $\mathcal{V}$. Using $P_m
\theta_n = {\tilde \theta}_n, ~~P_m \theta_m = {\tilde \theta}_m$,
\begin{eqnarray*}
I_{1,3} &=& \frac{2\pi}{L_m} \int_0^{2\pi} D {\tilde \theta}_{nm}
P_m D^2 \left ( U_n - \frac{2 \pi^2 \sigma}{L_n^2} \mathcal{H} [ {
\theta}_{n,\alpha \alpha} ] - U_m + \frac{2 \pi^2 \sigma}{L_m^2}
\mathcal{H} [
{\theta}_{m,\alpha \alpha}] \right ) \\
&+& \frac{4 \pi^3 \sigma}{L_m L_n^2} \int_0^{2 \pi} D {\tilde
\theta}_{nm} D^2 \mathcal{H} [{\tilde \theta}_{nm,\alpha\alpha}] -
\frac{4 \pi^3 \sigma}{L_m} \left ( \frac{1}{L_n^2} - \frac{1}{L_m^2}
\right ) \int_0^{2\pi} D^2 {\tilde \theta}_{nm} D \mathcal{H}
[{\tilde \theta}_{m,\alpha\alpha} ] d\alpha.
\end{eqnarray*}
Integrating by parts the second term in $I_{1,3}$ above and using
Lemma \ref{lemUTtheta} again, we obtain
$$ |I_{1,3} | \le - \frac{4 \pi^3 \sigma}{L_m L_n^2} \| {\tilde \theta}_{nm}
\|_{5/2} + C \epsilon E_{mn} +  C \epsilon E_{mn}^{1/2} \| {\tilde
\theta}_{nm}\|_{2}, $$ where $C$ depends on the diameter of
$\mathcal{V}$. Now using Lemma \ref{lemUTtheta}, we obtain
\begin{eqnarray*}
\frac{d(L_n-L_m)^2}{dt} &=& 2 (L_m-L_n) \int_0^{2 \pi} \left [ (U_n
- U_m ) (1+\theta_{n,\alpha}) + U_m (\theta_{n,\alpha} -
\theta_{m,\alpha} )\right ]
d\alpha \\
&\leq& c E_{nm}^{1/2} \Big ( \epsilon\| U_n - U_m \|_0 +\left\| U_n
-\frac{2\pi^2\sigma}{L_n^2}\mathcal{H}[\theta_{n,\alpha\alpha}]-
U_m+\frac{2\pi^2\sigma}{L_m^2}\mathcal{H}[\theta_{m,\alpha\alpha}]
\right\|_0\\
&& + \| {\tilde \theta}_{nm} \|_1 \| U_m \|_0 \Big ) \le C\epsilon
\left(E_{nm}+E_{nm}^{1/2}\|\tilde\theta_{nm}\|_2\right),
\end{eqnarray*}
$C$ depends on the diameter of $\mathcal{V}$. So for $I_2$, we use
the same method as we did for $I_1$ and  combine all the terms. So
we  obtain
\begin{eqnarray*}
\frac{dE_{mn}}{dt}&\leq& - \frac{4 \pi^3 \sigma}{L_m L_n^2} \|
{\tilde \theta}_{nm} \|_{5/2}^2 + \frac{4 \pi^3 \sigma}{L_m L_n^2}
\| {\tilde \theta}_{nm} \|_{3/2}^2 + c \epsilon E_{mn}^{1/2}
\|{\tilde \theta}_{nm} \|_{2}
+ \frac{c}{n} \epsilon E_{mn}^{1/2} + c\epsilon E_{mn} \\
&\le& - \frac{3 \pi^3 \sigma}{2(2\pi + 1)^3} \| {\tilde \theta}_{nm}
\|_{5/2}^2 + \frac{c}{2} \epsilon \|{\tilde \theta}_{nm} \|^2_{2} +
\frac{c}{n} E_{mn}^{1/2} + c_1 \epsilon E_{mn},
\end{eqnarray*}
where $c$ and $c_1$ depends on the diameter of $\mathcal{V}$.  Since
$\| {\tilde \theta}_{nm} \|_{5/2} \ge \| {\tilde \theta}_{nm} \|_2
$, it follows that for $\epsilon$ sufficiently small
$$ - \frac{3 \pi^3 \sigma}{ (2\pi+1)^3} \| {\tilde \theta}_{nm} \|_{5/2}^2
+ c\epsilon \|{\tilde \theta}_{nm} \|^2_{2} \le 0. $$ So,
\begin{eqnarray}
\frac{dE_{mn}}{dt}\leq c E_{mn}+ \frac{c}{n}E_{mn}^{1/2}.\nonumber
\end{eqnarray}
This can be restated as
$$\frac{dE_{mn}^{1/2}}{dt}\leq c E_{mn}^{1/2}+ \frac{c}{n}.$$
We solve the differential inequality to see that
$$ E_{mn}^{1/2}(t)\leq E_{mn}^{1/2} (0)e^{ct}+\frac{1}{n}(e^{ct}-1).$$
Since
$$ E_{mn}(0)=E_{mn}^1(0)\leq \frac{c}{n^2}\|\tilde\theta_0\|_r^2,$$
we have
$$ E_{mn}^{1/2}(t)\leq \frac{c}{n}(\|\tilde\theta_0\|_r+1)e^{ct}.$$
Thus, solutions do form a Cauchy sequence in
$C\left([0,S];\dot{H}^1\times\mathbb{R}\right)$.
\end{proof}

\begin{remark}
We now know that the solutions of the initial value problem
(\ref{2.12}), $(\tilde\theta_n,L_n)$, approach a limit as
$n\rightarrow \infty$ in
$C\left([0,S];\dot{H}^1\times\mathbb{R}\right)$. Call this limit
$X=(\tilde\theta,L)$.
\end{remark}
\begin{note}
\label{note5.3}
By Proposition \ref{prop2.14}, we know that $\|\tilde\theta_n(\cdot,t)\|_r\leq\|\mathcal{Q}_1\theta_0\|_r$ for all $t\ge 0$. Since $\dot{H}^r$ is a Hilbert space, its unit ball is weakly compact. Thus, $\tilde\theta_n\rightharpoonup\tilde\theta$  in $\dot{H}^r$. Furthermore, by Fatou's Lemma, we also have
$$\|\tilde\theta\|_r\leq\liminf_{n\rightarrow\infty}\|\tilde\theta_n\|_r\leq\|\mathcal{Q}_1\theta_0\|_r.$$
\end{note}

\begin{lemma}
\label{lem5.3} For $r\ge3$, there exists sufficiently small ball
size $\epsilon$ of  $ \mathcal{B}$ such that  as
$n\rightarrow\infty$, the limit of the initial value problem
(\ref{2.12}) $X=(\Tilde\theta,L)\in C\left((0,S];\mathcal{V}\right)$
for any $S>0$.
\end{lemma}
\begin{proof}

 Note that estimates in Corollary
\ref{coro4.2} and Proposition \ref{prop4.2}. Since $L_n \in \left (2
\pi -1 , 2 \pi +1 \right )$, we have
$$\frac{dE_n}{dt}\leq -\frac{\sigma}{9}\|\tilde\theta_n\|_{r+3/2}^2.$$
 It implies
$$  \frac{1}{2}E_n (t)+\frac\sigma9\int_0^t \| {\tilde \theta}_n \|_{r+3/2}^2 dt \le \frac{1}{2}E_n (0) \leq\frac12\|\mathcal{Q}_1\theta_0\|_r.$$
Hence ${\tilde \theta}_n $ is a bounded sequence in $ L^2 \left( [0,
\infty), \dot{H}^{r+3/2} \right )$. So, there exists a subsequence that
converges weakly, and it is easily argued that the limit can only be
${\tilde \theta}$. This means that for any interval $(0, S')$ there
exists $S_0$ in that interval so that $\| {\tilde \theta} (., S_0)
\|_{r+3/2} < \infty$. Now consider the solution to (B.1), (B.3) and (B.4) with $S_0$ as  initial time. In particular, ${\tilde \theta} (., S_0) \in
\dot{H}^{r+1}\cap\mathcal{B}$. Taking ${\tilde \theta} (., S_0)$ as
initial data in $\dot{H}^{r+1}\cap\mathcal{B}$, repeating the proof
of Proposition \ref{prop2.13} with $r+1$ instead of $r$, and by
Corollary \ref{coro4.2} and Proposition \ref{prop4.2},  we have global solutions
$\tilde\theta_n^{S_0}\in C^1 \left ( [S_0, \infty), \dot{H}^{r+1}
\cap\mathcal{B} \right) $ for sufficiently small $\epsilon$.
Again, by uniqueness of solutions to the approximate equation (\ref{2.12}) (Proposition \ref{prop2.13}), these solutions are identical to $\tilde\theta_n$ in their intervals of existence.
Also,  by Proposition 4.3, we have
\begin{equation}
\label{bound_rn}\|\tilde\theta_n(\cdot,t)\|_{r+1}\leq \|\tilde\theta_n(\cdot,S_0)\|_{r+1}e^{-\frac{\sigma}{36}(t-S_0)}\leq
\|\tilde\theta(\cdot,S_0)\|_{r+1}e^{-\frac{\sigma (t-S_0)}{36}}, \mbox{ for all }t\geq S_0.
\end{equation}

 From interpolation theorem in Sobolev space, we have
\begin{equation}
\label{interpolation}
\|\Tilde\theta_m-\tilde\theta_n\|_s\leq C
\|\Tilde\theta_m-\tilde\theta_n\|_0^{1-\frac{s}{r+1}}\|\Tilde\theta_m-\tilde\theta_n\|_{r+1}^{\frac{s}{r+1}}.
\end{equation}
By Lemma \ref{lem5.2} and (\ref{bound_rn}), we know that the right
side of (\ref{interpolation}) goes to zero uniformly on $[S_0,S]$,
as $n,m\rightarrow\infty$ for any $1\le s<r+1$. This implies $X\in
C\left([S_0,S];\dot{H}^{s}\times \mathbb{R}\right)$.
  Since the choice of
$S'$ is arbitrarily small, it follows that ${\tilde \theta} \in C
\left( (0, S], \dot{H}^{r} \right ) $.
\end{proof}

\begin{prop}\label{prop5.4} (continuity at $t=0$ in $\dot{H}^r$) For $r\ge3$, we have
\begin{equation}
\label{cont0} \lim_{t\rightarrow
0+}\|\tilde\theta(\cdot,t)-\mathcal{Q}_1\theta_0\|_r=0.
\end{equation}
\end{prop}
\begin{proof}
Replacing $r+1$ by $r$ in (\ref{interpolation}),  using the uniform bound of $\tilde\theta_n$ in $\dot{H}^r$ and $\tilde\theta_n\in C^1\left([0, \infty);\dot{H}^r\right)$, we find that $\tilde\theta_n\rightarrow\tilde\theta$ in $C\left([0,S];\dot{H}^s\right)$ as $n\rightarrow\infty$ for any $S>0$, where $1\leq s<r$.

Let $\eta>0$ and
$\phi\in H^{-r}\big(\mathbb{T}[0,2\pi]\big)$. For any $s$
satisfying $1\leq s<r$, choose $\varphi\in
H^{-s}\big(\mathbb{T}[0,2\pi]\big)$ so that
\begin{eqnarray}
\label{5.3}
\|\phi-\varphi\|_{-r}\leq \frac{\eta}{3}.
\end{eqnarray}
We know that such a $\varphi$ can be found since
$H^{-s}\big(\mathbb{T}[0,2\pi]\big)$ is dense in
$H^{-r}\big(\mathbb{T}[0,2\pi]\big)$. We have
\begin{eqnarray}
\label{5.4} \langle\phi,\Tilde\theta_n\rangle-\langle
\phi,\Tilde\theta\rangle=\langle\phi-\varphi,\Tilde\theta_n\rangle
+\langle \varphi-\phi,\Tilde\theta\rangle
+\langle\varphi,\Tilde\theta_n-\Tilde\theta\rangle,
\end{eqnarray}
where $\langle\cdot,\cdot\rangle$ denotes the pairing with dual
spaces. The first two terms can be bounded by $\frac{\eta}{3}$ using
(\ref{5.3}) and  uniform bounds on $\Tilde\theta$ and
$\tilde\theta_n$ in $\dot{H}^r$. For the third term, we choose $n$
large enough so that $\|\tilde\theta-\tilde\theta_n\|_s\leq \eta/3$.
Thus, (\ref{5.4}) is bounded by $\eta$. Since $\eta$ is arbitrary
and these bounds are  uniform in time,  we conclude that
$\tilde\theta\in C_W\left([0,S];\dot{H}^r\right)$. To prove the
lemma, it is enough to show $\lim_{t\rightarrow
0+}\|\tilde\theta(\cdot,t)\|_r=\|\mathcal{Q}_1\theta_0\|_r=0$.

By Note \ref{note5.3}, we
 know
$\|\Tilde\theta(\cdot,t)\|_r\leq\|\mathcal{Q}_1\theta_0\|_r$. This
means $\limsup_{t\rightarrow
0+}\|\Tilde\theta(\cdot,t)\|_r\leq\|\mathcal{Q}_1\theta_0\|_r$. From
the fact that $\tilde\theta\in C_W\big([0,S];\dot{H}^r\big)$, we
have $\liminf_{t\rightarrow 0+}\|\tilde\theta(\cdot,t)\|_r\geq
\|\mathcal{Q}_1\theta_0\|_r$. Hence, (\ref{cont0}) holds. This gives
us strong right continuity at $t=0$.
\end{proof}
By Lemma \ref{lem5.3} and Proposition \ref{prop5.4}, we have
\begin{corollary}\label{note5.5} For $r\ge 3$, there exists sufficient small ball size $\epsilon$ of $\mathcal{B}$  such  that $X\in C\left([0,S];\mathcal{V}\right)$ for any $S>0$.
\end{corollary}

\begin{prop}\label{coro5.6} For $r\ge4$,
$X$ is a classical solution to the initial value problem (B.1),
(B.3) and (B.4) with initial condition (\ref{2.4}) for any $S>0$,
where $\tilde\theta\in
C\left([0,S];C^3\left(\mathbb{T}[0,2\pi]\right)\right)\cap
C^1\big([0,S];C\left(\mathbb{T}[0,2\pi]\right)\big)$ and $L\in
C^1[0,S]$.
\end{prop}
\begin{proof}
For $r\ge4$, by Sobolev embedding theorem and Corollary
\ref{note5.5}, we know $X\in
C\left([0,S];C^3\left(\mathbb{T}[0,2\pi]\right)\times\mathbb{R}\right)$
and $\tilde\theta_n\rightarrow \tilde\theta$ as $n\rightarrow\infty$
in $ C\left([0,S];C^3\left(\mathbb{T}[0,2\pi]\right)\right)\cap
C\left([0,S];\dot{H}^s\right)$, for $1\le s<r$.

 Since $g$ is
$C^1$ in the open ball $\dot{H}^1$, $ g(\Tilde\theta_n)\rightarrow
g(\Tilde\theta)\mbox{ as }n\rightarrow \infty$. So
$\hat{\theta}(1;t)=g(\Tilde\theta)$ and $\tilde\theta$ satisfy
(B.4). By Proposition \ref{prop2.5} and (\ref{mathcalF}), we see
that both $\{\gamma_n\}_{n=2}^{\infty}$ and
$\left\{\mathcal{F}[\omega_n]\gamma_n\right\}_{n=2}^{\infty}$ are
Cauchy sequences in
$C\left([0,S];H^1\left(\mathbb{T}[0,2\pi]\right)\right)$. Hence, it
allows us to
 pass to the limit as $n\rightarrow\infty$ in the equation
$$\left(I+\mathcal{F}[\omega_n]\right)\gamma_n=\frac{2\pi}{L_n}\theta_{n,\alpha\alpha},$$
  and obtain
$$\big(I+\mathcal{F}[\omega]\big)\gamma=\frac{2\pi}{L}\theta_{\alpha\alpha}.$$
By  Proposition \ref{prop2.5} again , we have $\gamma\in
C\left([0,S];H^{r-2}\left(\mathbb{T}[0,2\pi]\right)\right)$. We also
have
\begin{equation} \label{limit_theta}\Tilde\theta_n(\alpha,t)=P_n
\theta_0(\alpha)+ \int_0^tF_{n,1}\big(X_n(t')\big)dt'.
\end{equation} From Lemma \ref{lemUTtheta},  it follows that
$\left\{F_{n,1}\right\}_{n=2}^\infty$ is a Cauchy sequence in
$C\left([0,S];\dot{H^0}\right)$. Replacing $r+1$ by $r-3$ and
$\tilde\theta_n$ by $F_{n,1}$ in (\ref{interpolation}) with the
uniform bound of $F_{n,1}$ in $\dot{H}^{r-3}$, we see
$\left\{F_{n,1}\right\}_{n=2}^\infty$ is a Cauchy sequence in
$C\left([0,S];\dot{H^s}\right)$ for $0\le s<r-3$. Hence, we take the
limit in (\ref{limit_theta}), yielding
\begin{eqnarray*}
\Tilde\theta(\alpha,t)=\mathcal{Q}_1\theta_0(\alpha)+\int_0^tF^1\big(X(t')\big)dt',
\end{eqnarray*}
where $F^1$ is the right-hand side of the first equation in (B.1).
This is differentiable in time, giving $\Tilde\theta_t={F}^1(X)\in
C\left([0,S];C\left(\mathbb{T}[0,2\pi]\right)\right)$. Similarly,
$L$ satisfies the second equation of (B.1) and $L_t\in C[0,S]$.
Thus, $X$ is a classical solution to (B.1), (B.3) and (B.4) with
initial condition (\ref{2.4}).
\end{proof}
\begin{lemma}
\label{lem5.4} For $r\ge3$, there exists sufficiently small  ball
size $\epsilon$ of $\mathcal{B}$ such that if $X^{(1)}\in
\mathcal{V}$
 and $X^{(2)}\in \mathcal{V}$ are solutions to
 the initial value problem (B.1), (B.3) and (B.4) with initial condition (\ref{2.4})
for the interval of time $[0,S]$ with any $S>0$, and
 the corresponding initial data $X^{(1)}(\alpha,0)\in \mathcal{V}$
 and $X^{(2)}(\alpha,0)\in \mathcal{V}$,
 then for $0\leq t\leq S$,
\begin{multline*}
\left\|\Tilde\theta^{(1)}(\cdot,t)-
\Tilde\theta^{(2)}(\cdot,t)\right\|_1+\left|L^{(1)}(t)-L^{(2)}(t)\right|\\\leq
\left(\left\|\Tilde\theta^{(1)}(\cdot,0)-
\Tilde\theta^{(2)}(\cdot,0)\right\|_1+\left|L^{(1)}(0)-L^{(2)}(0)\right|\right)\exp\{Bt\}.\end{multline*}
\end{lemma}
\begin{proof} This proof is very similar to the proof of
Lemma  \ref{lem5.2}, and we re-use some notation. Define $E_d$, the
energy function for the difference of two solutions, by
$E_d^1+(L^{(1)}-L^{(2)})^2$. Here,
$$E_d^1=\frac{1}{2}\int_0^{2\pi}(D_\alpha( \Tilde\theta^{(1)}- \Tilde\theta^{(2)}))^2d\alpha.$$
We now wish to estimate how this energy changes over time.
\begin{eqnarray}
\frac{dE_d^1}{dt}&=&\int_0^{2\pi}D_\alpha(\Tilde\theta^{(1)}- \Tilde\theta^{(2)})D_\alpha (\tilde\theta_{t}^{(1)}-\tilde\theta^{(2)}_{t})d\alpha\nonumber\\
&=&\int_0^{2\pi}D_\alpha(\Tilde\theta^{(1)}-
\Tilde\theta^{(2)})D_\alpha^2\mathcal{Q}_1\big(\frac{2\pi}{L^{(1)}}
U^{(1)}-
\frac{2\pi}{L^{(2)}} U^{(2)}\big)d\alpha\nonumber\\
&&+\int_0^{2\pi}D_\alpha(\Tilde\theta^{(1)}-
\tilde\theta^{(2)})D_\alpha \mathcal{Q}_1\Big(\frac{2\pi}{L^{(1)}}
\big(T^{(1)}(1+\theta_{\alpha}^{(1)})\big)-\frac{2\pi}{L^{(2)}}
 \big(T^{(2)}(1+\theta^{(2)}_{\alpha})\big)\Big)d\alpha.\nonumber
\end{eqnarray}
  Using  the same estimates  as that in Lemma \ref{lem5.2}, we have
\begin{eqnarray}
\frac{dE_d^1}{dt}\leq-\frac{4\pi^3}{(L^{(2)})^3}\sigma\Big(\big\|\Tilde\theta^{(1)}-
\Tilde\theta^{(2)}\big\|_{5/2}^2-\big\|\Tilde\theta^{(1)}-
\Tilde\theta^{(2)}\big\|_{3/2}^2\Big) +
c\epsilon\big(\big\|\Tilde\theta^{(1)}-\Tilde\theta^{(2)}\big\|_2^2+E_d\big),\nonumber
\end{eqnarray}with $c$ depends on the diameter of  $\mathcal{V}$.
 We also have
\begin{eqnarray}
\frac{d(L^{(1)}-L^{(2)})^2}{dt}
&\leq&c\epsilon\big(\|\Tilde\theta^{(1)}-
\Tilde\theta^{(2)}\|_2^2+E_d\big),\nonumber
\end{eqnarray}with $c$ depends on the diameter of $\mathcal{V}$. As
what we did in Lemma \ref{lem5.2}, for sufficiently small
$\epsilon$, there exists a positive constant $B$ such that
\begin{eqnarray}
\frac{dE_d}{dt}\leq BE_d.\nonumber
\end{eqnarray}
We solve the differential inequality to see that
$$ E_d(t)\leq E_d(0)e^{Bt}.$$
This proves the theorem.
\end{proof}
Hence,  uniqueness follows from Lemma \ref{lem5.4}.
\begin{lemma}\label{uniqueness}  For $r\ge 3$, there exists sufficiently small  ball size $\epsilon$
of $\mathcal{B}$ such that solution $X=\big(\Tilde\theta,L
\big)\in\mathcal{V}$ to  (B.1) , (B.3) and (B.4) with initial
condition (\ref{2.4}) is unique in $\dot{H}^1\times\mathbb{R}$.
\end{lemma}

\noindent{\bf Proof of Lemma \ref{lem2.16}:} This follows from
Lemmas \ref{lem5.2}, \ref{lem5.3}, \ref{uniqueness}, Corollary
\ref{note5.5} and Proposition \ref{coro5.6}.

\noindent{\bf Proof of Proposition \ref{prop2.18}:} Since the two
flows are both irrotational and incompressible, there exist velocity
potential $\phi_{i}$ for two fluids, $i=1,2$.
 Taking the derivative with respect to $t$ on both sides of (\ref{5.5}),
 we have
\begin{eqnarray*}
\frac{d\mathcal{S}(t)}{dt}&=&\frac{1}{2}\Im\int_0^{2\pi}\big(z_{\alpha}z_t^{\ast}-z_tz_\alpha^\ast\big)d\alpha\nonumber\\
&=&-\frac{L}{4\pi}\Re\int_0^{2\pi}\big(ie^{i\frac{\pi}{2}+i\alpha+i\theta(\alpha)}
z_t^{\ast}+z_t(-ie^{-i\frac{\pi}2-i\alpha-i\theta(\alpha)})\big)d\alpha\\
&=&-\frac{L}{2\pi}\int_0^{2\pi}(x_t,y_t)\cdot{\bf
n}d\alpha=-\frac{L}{2\pi}\int_0^{2\pi}Ud\alpha=-\frac{L}{2\pi}\int_0^{2\pi}\frac{\partial\phi_2}{\partial
{\bf n}}d\alpha.\nonumber
 \end{eqnarray*}
 By Green's second identity, we have
 \begin{eqnarray}
\frac{d\mathcal{S}(t)}{dt}=0.\nonumber
\end{eqnarray}
Hence the area of the bubble is invariant with time.

 Since $\big(\tilde\theta_n,L_n\big)$ converges to
$(\Tilde\theta,L)$ in
$C\left((0,S];\dot{H}^r\times\mathbb{R}\right)$ for any $S>0$, by
Proposition \ref{prop2.14} and
$\|P_n\theta_0\|_r\leq\|\mathcal{Q}_1\theta_0\|_r$, we have
\begin{eqnarray}\label{energytheta}\|\Tilde\theta(\cdot,t)\|_r=
\lim_{n\rightarrow\infty}\|\tilde\theta_n(\cdot,t)\|_r\leq
\|\mathcal{Q}_1\theta_0\|_re^{-\frac{1}{36}\sigma t}.
\end{eqnarray}

By (\ref{2.7}) and (\ref{energytheta}), the statement for
$\hat{\theta}(\pm1)$ hold.

Since the area is invariant with time, we have
\begin{eqnarray}\label{5.11}
\frac{L^2}{8\pi^2}\Im\int_0^{2\pi}\omega_\alpha\omega^{\ast}d\alpha
=\mathcal{S}\frac{1}{2\pi}\Im\int_0^{2\pi}\omega_{0,\alpha}\omega^{\ast}_0d\alpha.
\end{eqnarray}
(\ref{5.11}) gives us
\begin{align*}
\big(L^2-4\pi\mathcal{S}\big)\Im\int_0^{2\pi}\omega_\alpha\omega^{\ast}d\alpha+4\pi\mathcal{S}\Big(\Im\int_0^{2\pi}\omega_\alpha\omega^{\ast}d\alpha
-\Im\int_0^{2\pi}\omega_{0,\alpha}\omega^{\ast}_0d\alpha\Big)=0.
\end{align*}
 It implies that
\begin{equation*}
L-2\sqrt{\pi\mathcal{S}}=-\frac{L^2}{2\pi(L+2\sqrt{\pi\mathcal{S}})}\Big(\Im\int_0^{2\pi}\omega_\alpha\omega^{\ast}d\alpha
-\Im\int_0^{2\pi}\omega_{0,\alpha}\omega^{\ast}_0d\alpha\Big).
\end{equation*}
Hence, using $2\pi-1<L<2\pi+1$, we induce the following estimate:
\begin{equation*}
|L-2\sqrt{\pi\mathcal{S}}|\leq C\|\tilde\theta\|_1
\end{equation*}
with $C$ depending on $\mathcal{S}$ and the diameter of
$\mathcal{B}$.  From (\ref{energytheta}), the result for $L$
follows.

From (B.2), using (\ref{T0}) and $2\pi-1<L<2\pi+1$, we have
\begin{equation}\label{5.12}\left|\hat{\theta}(0;t)-\hat{\theta}_0(0)\right|\leq
C\int_0^t\|T(\cdot,t')\|_0\left(1+\|\tilde\theta(\cdot,t')\|_1\right)dt\leq
C\int_0^t\|\tilde\theta(\cdot,t')\|_3dt.\end{equation} Hence,
plugging estimates (\ref{energytheta}) into (\ref{5.12}), the result
for $\hat{\theta}(0)$ holds.

\section{appendix}

Proof of Lemma \ref{lem3.3} (\cite{AD1}):
\begin{proof}
We note that
\begin{displaymath}
D_\alpha^k q_1[\omega]=\int_0^1 t^k D^k
\omega_{\alpha}(t\alpha+(1-t)\alpha')dt~~,~~ D_{\alpha'}^k
q_1[\omega]=\int_0^1 (1-t)^k D^k
\omega_{\alpha}(t\alpha+(1-t)\alpha')dt.
\end{displaymath}
Then, using $2 \pi$ periodicity of $D^k \omega_\alpha$, we obtain
\begin{eqnarray}
&&\int_a^{a+2\pi}\Big|\int_0^1 t^k D^k
\omega_\alpha(t\alpha+(1-t)\alpha')dt\Big|^2d\alpha'\nonumber\\
&\leq&\int_a^{a+2\pi}\Big(\int_0^1| D^k
\omega_\alpha(t\alpha+(1-t)\alpha')(1-t)^{1/4}|^2dt\Big)\Big(\int_0^1 t^{2k}
(1-t)^{-1/2} dt \Big)d\alpha'\nonumber\\
&\leq& C\int_0^1\int_a^{a+2\pi}| D^k
\omega_\alpha(t\alpha+(1-t)\alpha')(1-t)^{1/4}|^2d\alpha'dt\nonumber\\
&\leq& C \int_0^1 \int_{a(1-t)+t\alpha}^{(a+2\pi)(1-t)+t\alpha}| D^k
\omega_\alpha(u)|^2(1-t)^{-1/2}dudt \nonumber\\
&\le& C \int_0^1 (1-t)^{-1/2} dt
\int_0^{2 \pi} |D^k \omega_\alpha (u)|^2 du
\leq C\|D^k \omega_\alpha\|^2_0. \nonumber
\end{eqnarray}
So $D_\alpha^k q_1 \in H^k [a, a+2\pi]$ in variable $\alpha'$ and
$\| D_\alpha^k q_1 [\omega] \|_0 \le C \| \omega_\alpha \|_k $ with $C$
only dependent on $k$. Again
\begin{eqnarray}
&&\int_a^{a+2\pi}\Big|\int_0^1 (1-t)^k D^k
\omega_\alpha(t\alpha+(1-t)\alpha')dt\Big|^2d\alpha'\nonumber\\
&\leq&\int_a^{a+2\pi}\Big(\int_0^1| D^k
\omega_\alpha(t\alpha+(1-t)\alpha')(1-t)^{1/4}|^2dt\Big)\Big(\int_0^1
(1-t)^{2k-1/2} dt\Big)d\alpha'\nonumber\\
&\leq& C\int_0^1\int_a^{a+2\pi}| D^k
\omega_\alpha(t\alpha+(1-t)\alpha')(1-t)^{1/4}|^2d\alpha'dt\nonumber\\
&\leq& C \int_0^1 \int_{a(1-t)+t\alpha}^{(a+2\pi)(1-t)+t\alpha}| D^k
\omega_\alpha(u)|^2(1-t)^{-1/2}dudt
\leq C\|D^k \omega_\alpha\|^2_0. \nonumber
\end{eqnarray}
So $D_{\alpha'}^k q_1 \in H^k [a, a+2\pi]$ in variable $\alpha'$ and
$\| D_{\alpha'}^k q_1 [\omega] \|_0 \le C \| \omega_\alpha \|_k $ with
$C$ only dependent on $k$.

We note that for $k \ge 0$
$$D_\alpha^k q_2[\omega]=-\int_0^1 t^k (1-t) D^k
\omega_{\alpha\alpha}(t\alpha+(1-t)\alpha')dt,$$
$$D_{\alpha'}^k q_2[\omega]=-\int_0^1 (1-t)^{k+1} D^k
\omega_{\alpha\alpha}(t\alpha+(1-t)\alpha')dt.$$
Similar arguments as above leads to the stated bounds for $q_2$.

From symmetry of $q_1$, $q_2$ in $\alpha$ and $\alpha'$, clearly the same
results hold with respect to $\alpha$ instead of $\alpha'$ integration.
\end{proof}

Proof of Lemma \ref{lem3.8} (\cite{AD1}):
\begin{proof}
We begin by taking $r-2$ derivatives of $\mathcal{K}[\omega] f$.
\begin{eqnarray}
D^{r-2}_\alpha \mathcal{K}[\omega] f (\alpha) &=&D^{r-2}_\alpha
\frac{1}{2\pi i}\int_{\alpha-\pi}^{\alpha+\pi}
 f(\alpha')\Big{[}\frac{1}{\omega(\alpha)-z(\alpha')}-\frac{1}{2\omega_{\alpha}(\alpha')}\cot{\frac{1}{2}(\alpha-\alpha')}\Big{]}d\alpha'\nonumber\\
 &=& \frac{1}{2\pi i}\int_{\alpha-\pi}^{\alpha+\pi}
 f(\alpha')D^{r-2}_\alpha\Big{[}\frac{1}{\omega(\alpha)-\omega(\alpha')}
-\frac{1}{2\omega_{\alpha}(\alpha')}\cot{\frac{1}{2}(\alpha-\alpha')}\Big{]}d\alpha'\nonumber\\
 &=& \frac{1}{2\pi i}\int_{\alpha-\pi}^{\alpha+\pi}
 f(\alpha')D^{r-2}_\alpha\Big{[}\frac{1}{\omega(\alpha)-\omega(\alpha')}
-\frac{1}{\omega_{\alpha}(\alpha')(\alpha-\alpha')}\Big{]}d\alpha'\nonumber\\
 &&-\frac{1}{2\pi i}\int_{\alpha-\pi}^{\alpha+\pi}
 \frac{f(\alpha')}{2\omega_{\alpha}(\alpha')}D^{r-2}_\alpha
l\big(\frac{1}{2}(\alpha-\alpha')\big)d\alpha'\nonumber\\
 &=& P_1+P_2.\nonumber
\end{eqnarray}
Since the function $l(\beta)$ is analytical for $-\frac{\pi}{2}\leq
\beta \leq \frac{\pi}{2}$, it is easy to have
\begin{displaymath}
\|P_2\|_0\leq \frac{C}{L} \|f\|_0, \mbox{ where C depends on } r.
\end{displaymath}
Let us see $P_1$.
\begin{eqnarray}
P_1&=&  \frac{1}{2\pi i}\int_{\alpha-\pi}^{\alpha+\pi}
\frac{ f(\alpha')}{\omega_{\alpha}(\alpha')}D^{r-2}_\alpha\Big{[}
\frac{\omega_{\alpha}(\alpha')}{\omega(\alpha)-\omega(\alpha')}-\frac{1}{\alpha-\alpha'}\Big{]}d\alpha'\nonumber\\
&=&  \frac{1}{2\pi i}\int_{\alpha-\pi}^{\alpha+\pi} \frac{
f(\alpha')}{\omega_{\alpha}(\alpha')}D^{r-2}_\alpha\Big(\frac{q_2[\omega](\alpha',\alpha)}{q_1[\omega](\alpha',\alpha)}\Big)
d\alpha'.\nonumber
\end{eqnarray}
(\ref{3.1}) implies that $\big|q_1[\omega](\alpha,\alpha')\big|\geq
\frac{1}{4}$. So by  Lemma \ref{lem3.5}, we have
\begin{displaymath}
\|P_1\|_0 \leq  \frac{C_1}{L} \|f\|_0
\exp{(\frac{C_2}{L}\|\omega_\alpha\|_{r-1})}.
\end{displaymath}
Hence first result follows.
Taking $\alpha$-derivative $r-1$ times $\mathcal{K}[\omega] f$ and integrating
by parts once,
\begin{eqnarray*}
D^{r-1}_\alpha \mathcal{K}[\omega] f (\alpha) &=&D^{r-2}_\alpha
\frac{1}{2\pi i}\int_{\alpha-\pi}^{\alpha+\pi}
D_{\alpha'}\Big(\frac{ f(\alpha')}{\omega_\alpha(\alpha')}\Big)
\Big{[}\frac{\omega_\alpha(\alpha)}{\omega(\alpha)-\omega(\alpha')}-
\frac{1}{2}\cot{\frac{1}{2}(\alpha-\alpha')}\Big{]}d\alpha'\nonumber\\
&=&\frac{1}{2\pi i}\int_{\alpha-\pi}^{\alpha+\pi}
D_{\alpha'}\Big(\frac{ f(\alpha')}{\omega_\alpha(\alpha')}
\Big)D^{r-2}_\alpha\Big{[}\frac{\omega_\alpha(\alpha)}{\omega(\alpha)-\omega(\alpha')}
-\frac{1}{2}\cot{\frac{1}{2}(\alpha-\alpha')}\Big{]}d\alpha'\nonumber\\
&=& -\frac{1}{2\pi
i}\int_{\alpha-\pi}^{\alpha+\pi}D_{\alpha'}\Big(\frac{
f(\alpha')}{\omega_\alpha(\alpha')}\Big)
D^{r-2}_\alpha \left ( \frac{q_2 [\omega] (\alpha, \alpha') }{
q_1 [\omega] (\alpha, \alpha') } \right ) d\alpha' \\
&&-\frac{1}{2\pi i}\int_{\alpha-\pi}^{\alpha+\pi}
D_{\alpha'}\Big(
\frac{f(\alpha')}{2\omega_{\alpha}(\alpha')}\Big)D^{r-2}_\alpha
l\big(\frac{1}{2}(\alpha-\alpha')\big)d\alpha'.\nonumber
\end{eqnarray*}
Using Lemma \ref{lem3.5}, the
the second inequality follows from Cauchy-Schwartz inequality after
noting that $ \| D \left (\frac{f}{\omega_{\alpha}} \right ) \|_0
\le C \| f \|_1 \| \omega_\alpha \|_1 $
\end{proof}

Proof of Lemma \ref{lem3.10} (\cite{AD1}):
\begin{proof}
We begin by writing $[\mathcal{H},\psi]$ as an integral operator:
\begin{displaymath}
[\mathcal{H},\psi]
f(\alpha)=\frac{1}{2\pi}\int_{\alpha-\pi}^{\alpha+\pi}
f(\alpha')\big(\psi(\alpha')-\psi(\alpha)\big)\cot\Big(\frac{1}{2}(\alpha-\alpha')\Big)d\alpha'.
\end{displaymath}
We can write the kernel as
\begin{displaymath}
\Big(\frac{\psi(\alpha')-\psi(\alpha)}{\alpha-\alpha'}\Big)\Big((\alpha-\alpha')\cot\Big(\frac{1}{2}(\alpha-\alpha')\Big).
\end{displaymath}
The first part of this product is a divided difference, and the
 second part is an analytic function on the domain $[-\frac{\pi}{2},\frac{\pi}{2}]$. The lemma now follows
  from the Generalized Young's Inequality.
\end{proof}

\end{document}